\documentclass[12pt,oneside]{amsart}

\usepackage{color}

\definecolor{darkgreen}{rgb}{0,0.5,0}

\usepackage[pdftex,colorlinks=true,citecolor=darkgreen,%
            pdftitle={An Explicit Theory of Heights for Hyperelliptic Jacobians of Genus Three},%
            pdfauthor={Michael Stoll}]{hyperref}

\usepackage{amssymb}
\usepackage{enumerate}

\usepackage[all]{xy}

\newcommand{\Z}{{\mathbb Z}}
\newcommand{\Q}{{\mathbb Q}}
\newcommand{\R}{{\mathbb R}}
\newcommand{\C}{{\mathbb C}}
\newcommand{\F}{{\mathbb F}}
\newcommand{\BP}{{\mathbb P}}
\newcommand{\BA}{{\mathbb A}}
\newcommand{\To}{\longrightarrow}
\newcommand{\disc}{\operatorname{disc}}
\newcommand{\Pic}{\operatorname{Pic}}
\newcommand{\CC}{{\mathcal C}}
\newcommand{\CE}{{\mathcal E}}
\newcommand{\CJ}{{\mathcal J}}
\newcommand{\CK}{{\mathcal K}}
\newcommand{\CL}{{\mathcal L}}
\newcommand{\CO}{{\mathcal O}}
\newcommand{\CQ}{{\mathcal Q}}
\newcommand{\CT}{{\mathcal T}}
\newcommand{\CU}{{\mathcal U}}
\newcommand{\CV}{{\mathcal V}}
\newcommand{\CW}{{\mathcal W}}
\newcommand{\CX}{{\mathcal X}}
\newcommand{\eps}{\varepsilon}
\newcommand{\Sym}{\operatorname{Sym}}
\newcommand{\GL}{\operatorname{GL}}
\newcommand{\SL}{\operatorname{SL}}
\newcommand{\PSL}{\operatorname{PSL}}
\newcommand{\PGL}{\operatorname{PGL}}
\newcommand{\SO}{\operatorname{SO}}
\renewcommand{\O}{\operatorname{O}}
\newcommand{\half}{\frac{1}{2}}
\newcommand{\fD}{\mathfrak D}
\newcommand{\fm}{\mathfrak m}
\newcommand{\fW}{\mathfrak W}
\newcommand{\Res}{\operatorname{Res}}
\newcommand{\Tr}{\operatorname{Tr}}
\newcommand{\even}{{\text{\rm even}}}
\newcommand{\vd}{{\underline{\delta}}}
\newcommand{\vtau}{{\underline{\tau}}}
\newcommand{\vxi}{{\underline{\xi}}}
\newcommand{\vx}{{\underline{x}}}
\newcommand{\vz}{{\underline{\zeta}}}
\newcommand{\wt}{\operatorname{wt}}

\newcommand{\good}{{\text{\rm good}}}
\newcommand{\pr}{\operatorname{pr}}

\newtheorem{Theorem}{Theorem}[section]
\newtheorem{Lemma}[Theorem]{Lemma}
\newtheorem{Proposition}[Theorem]{Proposition}
\newtheorem{Corollary}[Theorem]{Corollary}

\theoremstyle{definition}
\newtheorem{Definition}[Theorem]{Definition}
\newtheorem{Example}[Theorem]{Example}
\newtheorem{Remark}[Theorem]{Remark}

\numberwithin{equation}{section}

\renewcommand{\star}{${}^\bigstar$}

\begin{document}

\title[Heights for Genus 3 Hyperelliptic Jacobians]%
      {An Explicit Theory of Heights \\ for Hyperelliptic Jacobians of Genus Three}

\author{Michael Stoll}
\address{Mathematisches Institut,
         Universit\"at Bayreuth,
         95440 Bayreuth, Germany.}
\email{Michael.Stoll@uni-bayreuth.de}
\date{\today}

\begin{abstract}
  We develop an explicit theory of Kummer varieties associated to Jacobians of
  hyperelliptic curves of genus~$3$, over any field~$k$ of characteristic~$\neq 2$.
  In particular, we provide explicit equations defining the Kummer variety~$\CK$
  as a subvariety of~$\BP^7$, together with explicit polynomials giving the
  duplication map on~$\CK$. A careful study of the degenerations of this map
  then forms the basis for the development of an explicit theory of heights
  on such Jacobians when $k$ is a number field. We use this input to obtain
  a good bound on the difference between naive and canonical height, which is
  a necessary ingredient for the explicit determination of the Mordell-Weil group.
  We illustrate our results with two examples.
\end{abstract}

\subjclass[2010]{14H40, 14H45, 11G10, 11G50, 14Q05, 14Q15}
\keywords{Kummer variety, hyperelliptic curve, genus 3, canonical height}

\maketitle


\section{Introduction}

The goal of this paper is to take up the approaches used to deal with
Jacobians and Kummer surfaces of curves of genus~$2$ by
Cassels and Flynn~\cite{CasselsFlynn} and by the author~\cite{StollH1,StollH2}
and extend them to hyperelliptic curves of genus~$3$. We always assume
that the base field~$k$ has characteristic~$\neq 2$.
A hyperelliptic curve~$\CC$ over~$k$ of genus~$3$ is then
given by an equation of the form $y^2 = f(x)$, where $f$ is a squarefree
polynomial of degree $7$ or~$8$ with coefficients in~$k$;
we take $\CC$ to be the smooth projective curve determined by this affine equation.
We denote the Jacobian variety of~$\CC$ by~$\CJ$.
Identifying points with their negatives on~$\CJ$, we obtain the Kummer variety
of~$\CJ$. It is known that the morphism $\CJ \to \BP^7$ given by the
linear system $|2\Theta|$ on~$\CJ$ (where $\Theta$ denotes the theta divisor)
induces an isomorphism of the Kummer variety with the image of~$\CJ$ in~$\BP^7$;
we denote the
image by $\CK \subset \BP^7$. Our first task is to find a suitable basis
of the Riemann-Roch space~$L(2 \Theta)$ and to give explicit equations defining~$\CK$,
thereby completing earlier work by Stubbs~\cite{Stubbs}, Duquesne~\cite{Duquesne}
and M\"uller~\cite{MuellerThesis,MuellerG3}. To this end, we make use of
the canonical identification of~$\CJ$ with $\CX = \Pic^4(\CC)$ and realize
the complement of $\Theta$ in~$\CX$ as the quotient of an explicit $6$-dimensional
variety~$\CV$ in~$\BA^{15}$ by the action of a certain group~$\Gamma$. This allows us
to identify the ring of regular functions on~$\CX \setminus \Theta$
with the ring of $\Gamma$-invariants in the coordinate ring of~$\CV$.
In this way, we obtain a natural basis of~$L(2\Theta)$, and we find the
quadric and the $34$~quartics that define~$\CK$; see Section~\ref{S:coords}.
We give the relation between the coordinates chosen here and those used
in previous work and discuss how transformations of the curve equation
induced by the action of~$\GL(2)$ on~$(x,z)$ act on our coordinates;
see Section~\ref{S:trans}. We then give a recipe that allows to decide
whether a $k$-rational point on~$\CK$ comes from a $k$-rational point on~$\CJ$
(Section~\ref{S:lift}).

The next task is to describe the maps $\CK \to \CK$ and $\Sym^2 \CK \to \Sym^2 \CK$
induced by multiplication by~$2$ and by $\{P, Q\} \mapsto \{P+Q, P-Q\}$
on~$\CJ$. We use the approach followed in~\cite{StollH1}: we consider the
action of a double cover of the $2$-torsion subgroup $\CJ[2]$ on the coordinate
ring of~$\BP^7$. This induces an action of~$\CJ[2]$ itself on forms of even
degree. We use the information obtained on the various eigenspaces and the
invariant subspaces in particular to obtain an explicit description of the duplication
map~$\vd$ and of the sum-and-difference map on~$\CK$. The study of the action
of~$\CJ[2]$ is done in Sections \ref{S:2tors} and~\ref{S:action}; the results
on the duplication map and on the sum-and-difference map are obtained in
Sections \ref{S:dup} and~\ref{S:sumdiff}, respectively. In Section~\ref{S:further},
we then study the degeneration of these maps that occur when we allow the
curve to acquire singularities. This is relevant in the context of bad reduction
and is needed as input for the results on the height difference bound.

We then turn to the topic motivating our study, which is the canonical
height $\hat{h}$ on the Jacobian, and, in particular, a bound on the
difference $h - \hat{h}$ between naive and canonical height. Such a bound
is a necessary ingredient for the determination of generators of the Mordell-Weil
group~$\CJ(k)$ (where $k$ now is a number field; in practice, usually $k = \Q$),
given generators of a finite-index subgroup. The difference $h - \hat{h}$
can be expressed in terms of the local `loss of precision' under~$\vd$ at
the various primes of bad reduction and the archimedean places of~$k$.
In analogy with~\cite{StollH1}, we obtain an estimate for this local
`loss of precision' in terms of the valuation of the discriminant of~$f$.
This is one of the main results of Section~\ref{S:heights}, together with
a statement on the structure of the local `height correction function',
which is analogous to that obtained in~\cite[Thm.~4.1]{StollH2}.
These results allow us to obtain reasonable bounds for the height difference.
We illustrate this by determining generators
of the Mordell-Weil group of the Jacobian of the curve $y^2 = 4x^7 - 4x + 1$.
We then use this result to determine the set of integral solutions of
the equation $y^2 - y = x^7 - x$, using the method of~\cite{BMSST};
see Section~\ref{S:example}.

In addition, we show in Section~\ref{S:twists} how one can obtain better bounds
(for a modified naive height) when the polynomial defining the curve is not primitive.
As an example, we determine explicit generators of the Mordell-Weil group
of the Jacobian of the curve given by the binomial coefficient equation
\[ \binom{y}{2} = \binom{x}{7} . \]

We have made available at~\cite{Files} files that can be read into Magma~\cite{Magma}
and provide explicit representations of the quartics defining the Kummer variety,
the matrices giving the action of $2$-torsion points, the polynomials defining
the duplication map and the matrix of bi-quadratic forms related to the
`sum-and-difference map'.

\subsection*{Acknowledgments}

I would like to thank Steffen M\"uller for helpful comments on a draft
version of this paper and for pointers to the literature.
The necessary computations were performed using the Magma computer
algebra system~\cite{Magma}. At~\cite{Files} we have made available the file
\texttt{Kum3-verification.magma}, which, when loaded into Magma,
will perform the computations necessary to verify a number of claims
made throughout the paper. These claims are marked by a star, like this\star.


\section{The Kummer Variety} \label{S:coords}

We consider a hyperelliptic curve of genus 3
over a field~$k$ of characteristic different from~$2$, given by the affine
equation
\[ \CC \colon y^2 = f_8 x^8 + f_7 x^7 + \ldots + f_1 x + f_0 = f(x) , \]
where $f$ is a squarefree polynomial of degree $7$ or~$8$.
(We do not assume that $\CC$ has a Weierstrass point at infinity, which
would correspond to $f$ having degree~$7$.)
Let $F(x,z)$ denote the octic binary form that is the homogenization of~$f$;
$F$ is squarefree.
Then $\CC$ has a smooth model in the weighted projective plane $\BP^2_{1,4,1}$
given by $y^2 = F(x,z)$. Here $x$ and~$z$ have weight~$1$ and $y$ has weight~$4$.
We denote the hyperelliptic involution on~$\CC$ by~$\iota$, so that
$\iota (x:y:z) = (x:-y:z)$.

As in the introduction, we denote the Jacobian variety of~$\CC$ by~$\CJ$.
We would like to find an explicit version  of the map
\[ \CJ \To \BP^7 \]
given by the linear system of twice the theta divisor; it
embeds the Kummer variety $\CJ/\{\pm 1\}$ into~$\BP^7$.
We denote the image by~$\CK$.

We note that the canonical class $\fW$ on~$\CC$ has degree~$4$. Therefore
$\CJ = \Pic^0_\CC$ is canonically isomorphic to $\CX = \Pic^4_\CC$, with
the isomorphism sending $\fD$ to~$\fD+\fW$. Then the map induced by~$\iota$
on~$\CX$ corresponds to multiplication by~$-1$ on~$\CJ$.
There is a canonical theta
divisor on~$\Pic^0_C$ whose support consists of the divisor classes of the form
$[(P_1) + (P_2)] - \fm$, where $\fm$ is the class of the polar
divisor~$(x)_\infty$; we have $\fW = 2 \fm$.
The support of the theta divisor is the locus of points on~$\CX$ that
are not represented by divisors in general position, where an effective
divisor$~\fD$ on~$\CC$ is \emph{in general position} unless
there is a point $P \in \CC$ such that $\fD \ge (P) + (\iota P)$.
This can be seen as follows. The image on~$\CX$ of a point $[(P_1) + (P_2)] - \fm$
on the theta divisor is represented by all effective divisors
of the form $(P_1) + (P_2) + (P) + (\iota P)$ for an arbitrary point $P \in \CC$.
If $P_2 \neq \iota P_1$, then the Riemann-Roch Theorem implies that
the linear system containing these divisors is one-dimensional, and so
\emph{all} divisors representing our point on~$\CX$ have this form;
in particular, there is no representative divisor in general position.
If $P_2 = \iota P_1$, then the linear system has dimension~$2$ and consists
of all divisors of the form $(P) + (\iota P) + (P') + (\iota P')$,
none of which is in general position.

We identify $\CJ$ and~$\CX$, and we denote the theta divisor on~$\CJ$
and its image on~$\CX$ by~$\Theta$. We write $L(n\Theta)$
for the Riemann-Roch space $L(\CX, n\Theta) \cong L(\CJ, n\Theta)$,
where $n \ge 0$ is an integer. It is known
that $\dim L(n \Theta) = n^3$. Since $\Theta$ is symmetric, the negation
map acts on~$L(n \Theta)$ (via $\phi \mapsto (P \mapsto \phi(-P))$),
and it makes sense to speak of even and odd functions in~$L(n \Theta)$
(with respect to this action). We write $L(n\Theta)^+$ for the subspace
of even functions. It is known that $\dim L(n\Theta)^+ = n^3/2 + 4$ for
$n$~even and $\dim L(n\Theta)^+ = (n^3+1)/2$ for $n$~odd.

We can parameterize effective degree~4 divisors in general position
as follows. Any such divisor~$\fD$ is given by a binary quartic form~$A(x,z)$
specifying the image of~$\fD$ on~$\BP^1$ under the hyperelliptic
quotient map $\pi \colon \CC \to \BP^1$, $(x:y:z) \mapsto (x:z)$, together with
another quartic binary form~$B(x,z)$ such that $y = B(x,z)$ on the
points in~$\fD$, with the correct multiplicity. (Note that by the
`general position' condition, $y$ is uniquely determined by $x$ and~$z$
for each point in the support of~$\fD$.) More precisely, we must have that
\begin{equation} \label{rel}
  B(x,z)^2 - A(x,z) C(x,z) = F(x,z)
\end{equation}
for a suitable quartic binary form~$C(x,z)$. We then have a statement
analogous to that given in~\cite[Chapter~4]{CasselsFlynn} for $\Pic^3$
of a curve of genus~$2$; see Lemma~\ref{RepLemma} below.
Before we can formulate it, we need some notation.

We let $Q$ be the ternary quadratic form $x_2^2 - x_1 x_3$. We write
\begin{equation} \label{eqn:D}
  D =  \begin{pmatrix} 0 & 0 & -1 \\ 0 & 2 & 0 \\ -1 & 0 & 0 \end{pmatrix}
\end{equation}
for the associated symmetric matrix (times~$2$) and
\[ \Gamma = \SO(Q) = \{\gamma \in \SL(3) : \gamma D \gamma^\top = D\} ; \]
then $-\Gamma = \O(Q) \setminus \SO(Q)$, and
$\pm\Gamma = \O(Q)$. We have the following elements in~$\Gamma$
(for arbitrary $\lambda$ and~$\mu$ in the base field):
\[  t_\lambda =
    \begin{pmatrix}
      \lambda & 0 & 0 \\ 0 & 1 & 0 \\ 0 & 0 & \lambda^{-1}
    \end{pmatrix}, \qquad
    n_\mu =
    \begin{pmatrix}
      1 & \mu & \mu^2 \\ 0 & 1 & 2\mu \\ 0 & 0 & 1
    \end{pmatrix} \quad\text{and}\quad
    w =
    \begin{pmatrix}
      0 & 0 & 1 \\ 0 & -1 & 0 \\ 1 & 0 & 0
    \end{pmatrix} ;
\]
these elements generate~$\Gamma$.

\begin{Lemma} \label{RepLemma}
  Two triples $(A,B,C)$ and $(A',B',C')$ satisfying~\eqref{rel} specify
  the same point on~$\CX$ if and only if $(A',B',C') = (A,B,C) \gamma$
  for some $\gamma \in \Gamma$.
  They represent opposite points (with respect to the involution on~$\CX$
  induced by~$\iota$) if and only if the relation above holds for
  some $\gamma \in -\Gamma$.
\end{Lemma}

\begin{proof}
  We first show that two triples specifying the same point are in the
  same $\Gamma$-orbit. Let $\fD$ and~$\fD'$ be the effective divisors
  of degree~4 given by $A(x,z) = 0$, $y = B(x,z)$ and by
  $A'(x,z) = 0$, $y = B'(x,z)$, respectively. By assumption, $\fD$ and~$\fD'$ are
  linearly equivalent, and they are both in general position.
  If $\fD$ and~$\fD'$ share a point~$P$ in their supports, then subtracting~$P$
  from both $\fD$ and~$\fD'$, we obtain two effective divisors of degree~3
  in general position that are linearly equivalent. Since such divisors
  are non-special, they must be equal, hence $\fD = \fD'$. So $A$ and~$A'$
  agree up to scaling, and $B' - B$ is a multiple of~$A$:
  \[ A' = \lambda A, \qquad B' = B + \mu A, \qquad
     C' = \lambda^{-1}(C + 2 \mu B + \mu^2 A) ;
  \]
  then $(A',B',C') = (A,B,C) n_\mu t_\lambda$.
  So we can now suppose that the supports of $\fD$ and~$\fD'$ are disjoint.
  Then, denoting by $\iota{\fD'}$ the image of~$\fD'$ under the hyperelliptic
  involution, $\fD + \iota{\fD'}$ is a divisor of degree~8 in general position,
  which is in twice the canonical class, so it is linearly equivalent to~$4 \fm$.
  Since the Riemann-Roch space of that
  divisor on~$\CC$ is generated (in terms of the affine coordinates obtained by
  setting $z = 1$) by $1, x, x^2, x^3, x^4, y$, there is a function of the form
  $y - \tilde{B}(x,1)$ with $\tilde{B}$ homogeneous of degree~4 that has divisor
  $\fD + \iota{\fD'} - 4\fm$.
  Equivalently, $\fD + \iota{\fD'}$ is the intersection
  of~$\CC$ with the curve given by $y = \tilde{B}(x,z)$. This implies
  that $\tilde{B}^2 - F$ is a constant times $A A'$. Up to scaling $A'$ and~$C'$
  by $\lambda$ and~$\lambda^{-1}$ for a suitable~$\lambda$ (this corresponds
  to acting on $(A',B',C')$ by $t_\lambda \in \Gamma$), we have
  \[ \tilde{B}^2 - A A' = F , \]
  so that $(A, \tilde{B}, A')$ corresponds to~$\fD$ and $(A', -\tilde{B}, A)$
  corresponds to~$\fD'$. The argument above (for the case $\fD = \fD'$) shows
  that $(A,B,C)$ and $(A, \tilde{B}, A')$ are in the same $\Gamma$-orbit,
  and the same is true of $(A',B',C')$ and $(A', -\tilde{B}, A)$.
  Finally,
  \[ (A', -\tilde{B}, A) = (A, \tilde{B}, A') w . \]

  Conversely, it is easy to see that the generators of~$\Gamma$ given
  above do not change the linear equivalence class of the associated divisor:
  the first two do not even change the divisor, and the third replaces
  $\fD$ by the linearly equivalent divisor~$\iota\fD'$, where $\fD + \fD' \sim 2\fW$
  is the divisor of $y - B(x,z)$ on~$\CC$.

  For the last statement, it suffices to observe that $(A,-B,C)$ gives
  the point opposite to that given by~$(A,B,C)$; the associated matrix
  is $-t_{-1} \in -\Gamma$.
\end{proof}

We write $A$, $B$, $C$ as follows.
\begin{align*}
  A(x,z) &= a_4 x^4 + a_3 x^3 z + a_2 x^2 z^2 + a_1 x z^3 + a_0 z^4 \\
  B(x,z) &= b_4 x^4 + b_3 x^3 z + b_2 x^2 z^2 + b_1 x z^3 + b_0 z^4 \\
  C(x,z) &= c_4 x^4 + c_3 x^3 z + c_2 x^2 z^2 + c_1 x z^3 + c_0 z^4
\end{align*}
and use $a_0, \dots, a_4, b_0 \dots, b_4, c_0, \dots, c_4$ as affine
coordinates on~$\BA^{15}$. We arrange these coefficients into a matrix
\begin{equation} \label{eqn:L}
  L =
   \begin{pmatrix}
     a_0 & a_1 & a_2 & a_3 & a_4 \\
     b_0 & b_1 & b_2 & b_3 & b_4 \\
     c_0 & c_1 & c_2 & c_3 & c_4
   \end{pmatrix} .
\end{equation}
Then $\gamma \in \pm\Gamma$ acts on~$\BA^{15}$ via multiplication
by~$\gamma^\top$ on the left on~$L$.
Since there is a multiplicative group sitting inside~$\Gamma$ acting by
$(A,B,C) \cdot \lambda = (\lambda A, B, \lambda^{-1} C)$,
any $\Gamma$-invariant polynomial must be a linear combination of monomials
having the same number of $a_i$ and $c_j$. Hence in any term of a homogeneous
$\Gamma$-invariant polynomial of degree~$d$, the number of factors~$b_i$ has
the same parity as~$d$. This shows that such a $\Gamma$-invariant polynomial
is even with respect to~$\iota$ if $d$ is even, and odd if $d$ is odd.

It is not hard to see that there are no $\Gamma$-invariant polynomials of degree~1:
by the above, they would have to be a linear combination of the~$b_i$,
but the involution $(A,B,C) \mapsto (C,-B,A) = (A,B,C) w$
negates all the~$b_i$. It is also not hard to check that the space of invariants
of degree~2 is spanned by the coefficients of the quadratic form
\[ B_l^2 - A_l C_l \in \Sym^2 \langle x_0, x_1, x_2, x_3, x_4 \rangle , \]
where
\begin{align*}
  A_l &= a_0 x_0 + a_1 x_1 + a_2 x_2 + a_3 x_3 + a_4 x_4 \\
  B_l &= b_0 x_0 + b_1 x_1 + b_2 x_2 + b_3 x_3 + b_4 x_4 \\
  C_l &= c_0 x_0 + c_1 x_1 + c_2 x_2 + c_3 x_3 + c_4 x_4
\end{align*}
are linear forms in five variables. We write
\[ B_l^2 - A_l C_l = \sum_{0 \le i \le j \le 4} \eta_{ij} x_i x_j , \]
so that $\eta_{ii} = b_i^2 - a_i c_i$ and for $i < j$,
$\eta_{ij} = 2 b_i b_j - a_i c_j - a_j c_i$. Up to scaling, the
quadratic form corresponds to the symmetric matrix
\begin{equation} \label{E:etamat}
  L^\top D L =
   \begin{pmatrix}
     2 \eta_{00} & \eta_{01} & \eta_{02} & \eta_{03} & \eta_{04} \\
     \eta_{01} & 2 \eta_{11} & \eta_{12} & \eta_{13} & \eta_{14} \\
     \eta_{02} & \eta_{12} & 2 \eta_{22} & \eta_{23} & \eta_{24} \\
     \eta_{03} & \eta_{13} & \eta_{23} & 2 \eta_{33} & \eta_{34} \\
     \eta_{04} & \eta_{14} & \eta_{24} & \eta_{34} & 2 \eta_{44}
   \end{pmatrix} ,
\end{equation}
and the image~$\CQ$ of the map $q \colon \BA^{15} \to \Sym^2 \BA^{5}$ given by this matrix
consists of the matrices of rank at most~$3$; it is therefore defined
by the $15$~different quartics obtained as $4 \times 4$-minors of this matrix.

Scaling $x$ by~$\lambda$ corresponds to scaling $a_j, b_j, c_j$ by~$\lambda^j$.
This introduces another grading on the coordinate ring of our~$\BA^{15}$;
we call the corresponding degree the \emph{weight}. We then have
$\wt(a_j) = \wt(b_j) = \wt(c_j) = j$ and therefore $\wt(\eta_{ij}) = i+j$.
The $15$~quartics defining~$\CQ$ have weights
\[ 12, 13, 14, 14, 15, 15, 16, 16, 16, 17, 17, 18, 18, 19, 20. \]
We will reserve the word \emph{degree} for the degree in terms of the~$\eta_{ij}$;
then it makes sense to set $\deg(a_j) = \deg(b_j) = \deg(c_j) = \half$.

We let $\CV \subset \BA^{15}$ be the affine variety
given by~\eqref{rel}. The defining equations of~$\CV$ then read
\[ \renewcommand{\arraystretch}{1.2}
   \begin{array}{r@{{}={}}l}
  b_0^2 - a_0 c_0 & f_0  \\
  2 b_0 b_1 - (a_0 c_1 + a_1 c_0) & f_1 \\
  2 b_0 b_2 + b_1^2 - (a_0 c_2 + a_1 c_1 + a_2 c_0) & f_2 \\
  2 b_0 b_3 + 2 b_1 b_2 - (a_0 c_3 + a_1 c_2 + a_2 c_1 + a_3 c_0) & f_3 \\
  2 b_0 b_4 + 2 b_1 b_3 + b_2^2
   - (a_0 c_4 + a_1 c_3 + a_2 c_2 + a_3 c_1 + a_4 c_0) & f_4 \\
  2 b_1 b_4 + 2 b_2 b_3 - (a_1 c_4 + a_2 c_3 + a_3 c_2 + a_4 c_1) & f_5 \\
  2 b_2 b_4 + b_3^2 - (a_2 c_4 + a_3 c_3 + a_4 c_2) & f_6 \\
  2 b_3 b_4 - (a_3 c_4 + a_4 c_3) & f_7 \\
  b_4^2 - a_4 c_4 & f_8 .
   \end{array}
\]
In terms of the $\eta_{ij}$, we have
\begin{gather*}
  \eta_{00} = f_0, \quad \eta_{01} = f_1, \quad \eta_{02} + \eta_{11} = f_2, \quad
  \eta_{03} + \eta_{12} = f_3, \quad \eta_{04} + \eta_{13} + \eta_{22} = f_4, \\
  \eta_{14} + \eta_{23} = f_5, \quad \eta_{24} + \eta_{33} = f_6, \quad
  \eta_{34} = f_7, \quad \eta_{44} = f_8 ;
\end{gather*}
in particular, the image of~$\CV$ under~$q$ is a linear `slice'~$\CW$
of~$\CQ$, cut out by the nine linear equations above (recall that the $\eta_{ij}$
are coordinates on the ambient space~$\Sym^2 \BA^{5}$ of~$\CQ$).
It is then natural to define $\deg(f_j) = 1$ and $\wt(f_j) = j$.

By Lemma~\ref{RepLemma}, the quotient $\CV/\Gamma$ of~$\CV$ by the action
of~$\Gamma$ can be identified with $\CU \mathrel{:=} \CX \setminus \Theta$,
the complement of the theta divisor in~$\CX$. Since the map~$q$ is given by $\pm \Gamma$-invariants,
we obtain a surjective morphism $\CK \setminus \kappa(\Theta) \to \CW$.
We will see that it is actually an isomorphism.

Functions in the Riemann-Roch space $L(n \Theta)$ will be represented
by $\Gamma$-invariant polynomials in the $a_i$, $b_i$, $c_i$.
Similarly, functions in the even part $L(n \Theta)^+$ of this space
are represented by
$\pm \Gamma$-invariant polynomials. A $\Gamma$-invariant polynomial
that is homogeneous of degree~$n$ in the $a_i$, $b_i$, $c_i$ will conversely
give rise to a function in~$L(n \Theta)$.
Modulo the relations defining~$\CV$,
there are six independent such invariants of degree~$2$. We choose
\[ \eta_{02},\; \eta_{03},\; \eta_{04},\; \eta_{13},\; \eta_{14},\; \eta_{24} \]
as representatives.
As mentioned above, invariants of even degree are $\pm\Gamma$-invariant
and so give rise to even functions on~$\CX$ with respect to~$\iota$,
whereas invariants of odd degree give rise to odd functions on~$\CX$.
Together with the constant function $1$, we have found seven functions
in~$L(2 \Theta) = L(2\Theta)^+$.
Since $\dim L(2\Theta) = 2^3 = 8$, we are missing one function. We will see
that is given by some quadratic form in the~$\eta_{ij}$ above, with the property
that it does not grow faster than them when we approach~$\Theta$.

To find this quadratic form,
we have to find out what $(\eta_{02} : \eta_{03} : \ldots : \eta_{24})$
tends to as we approach the point represented by $(x_1,y_1)+(x_2,y_2)+\fm$
on~$\CX$. A suitable approximation, taking $y = \ell(x)$ to be the
line interpolating between the two points,
\[ B(x,1) = \lambda (x-x_0)(x-x_1)(x-x_2) + \ell(x) , \]
$A_0(x) = (x-x_1)(x-x_2)$, $\varphi_{\pm}(x) = (f(x) \pm \ell(x)^2)/A_0(x)^2$,
$\psi(x) = \ell(x)/A_0(x)$, and
\[ A(x,1) = A_0(x)
              \bigl(\lambda^2 (x-x_0)^2
                      + \bigl(2 \lambda \psi(x_0) - \varphi_+'(x_0)\bigr)(x-x_0)
                      - \varphi_-(x_0)
                      + O(\lambda^{-1})\bigr) ,
\]
shows that\star\label{asymp}
\begin{align*}
  \eta_{02} &= -\lambda^2 (x_1 x_2)^2 + O(\lambda) \\
  \eta_{03} &= \lambda^2 (x_1+x_2) x_1 x_2 + O(\lambda) \\
  \eta_{04} &= -\lambda^2 x_1 x_2 + O(\lambda) \\
  \eta_{13} &= -\lambda^2 (x_1^2 + x_2^2) + O(\lambda) \\
  \eta_{14} &= \lambda^2 (x_1+x_2) + O(\lambda) \\
  \eta_{24} &= -\lambda^2 + O(1)
\end{align*}
as $\lambda \to \infty$. There are various quadratic expressions in these
that grow at most like~$\lambda^3$, namely
\begin{gather*}
  2 \eta_{04} \eta_{24} + \eta_{13} \eta_{24} - \eta_{14}^2, \quad
  \eta_{03} \eta_{24} - \eta_{04} \eta_{14}, \quad
  \eta_{02} \eta_{24} - \eta_{04}^2, \\
  \eta_{02} \eta_{14} - \eta_{03} \eta_{04}, \quad
  2 \eta_{02} \eta_{04} + \eta_{02} \eta_{13} - \eta_{03}^2
\end{gather*}
(they provide five independent even functions in $L(3\Theta)$ modulo~$L(2\Theta)$) and
\begin{equation} \label{etadef}
  \eta = \eta_{02} \eta_{24} - \eta_{03} \eta_{14}
            + \eta_{04}^2 + \eta_{04} \eta_{13} ,
\end{equation}
which in fact only grows like~$\lambda^2$ and therefore gives us the missing
basis element of~$L(2 \Theta)$. We find that\star
\[ \eta = \lambda^2 \frac{G(x_1,x_2) - 2 y_1 y_2}{(x_1-x_2)^2} + O(\lambda) , \]
where
\[ G(x_1, x_2) = 2 \sum_{j=0}^4 f_{2j} (x_1 x_2)^{j}
                  + (x_1 + x_2) \sum_{j=0}^3 f_{2j+1} (x_1 x_2)^{j} .
\]
(Note the similarity with the fourth Kummer surface coordinate
in the genus~$2$ case; see~\cite{CasselsFlynn}.)

The map $\CX \to \BP^7$ we are looking for is then given by
\[ (1 : \eta_{24} : \eta_{14} : \eta_{04} : \eta_{04} + \eta_{13}
      : \eta_{03} : \eta_{02} : \eta) .
\]
We use $(\xi_1, \ldots, \xi_8)$ to denote these coordinates (in the given order).
The reason for setting $\xi_5 = \eta_{04} + \eta_{13}$ rather than $\eta_{13}$
is that this leads to nicer formulas later on. For example, we then have
the simple quadratic relation
\begin{equation} \label{quadrel}
  \xi_1 \xi_8 - \xi_2 \xi_7 + \xi_3 \xi_6 - \xi_4 \xi_5 = 0 .
\end{equation}
Regarding degree and weight, we have, writing $\vxi = (\xi_1,\xi_2,\ldots,\xi_8)$, that
\[ \deg(\vxi) = (0,1,1,1,1,1,1,2) \quad\text{and}\quad
   \wt(\vxi) = (0,6,5,4,4,3,2,8) .
\]

It is known that the image~$\CK$ of the Kummer variety in~$\BP^7$
of a generic hyperelliptic Jacobian of genus~$3$
is given by a quadric and $34$ independent quartic relations that are not multiples
of the quadric; see~\cite[Thm.~3.3]{MuellerG3}. (For this, we can work over an
algebraically closed field, so that we can change coordinates to move
one of the Weierstrass points to infinity so that we are in the setting
of~\cite{MuellerG3}.) The quadric is just~\eqref{quadrel}.
It is also known~\cite[Prop.~3.1]{MuellerG3} that $\CK$ is defined by
quartic equations. Since there are $36$~quartic multiples of the quadric~\eqref{quadrel},
the space of quartics in eight variables has dimension~$330$ and the
space~$L(8\Theta)^+$ has dimension~$260$, there must be at least $34$~further
independent quartics vanishing on~$\CK$: the space of quartics vanishing
on~$\CK$ is the kernel of $\Sym^4 L(2 \Theta) \to L(8 \Theta)^+$, which
has dimension $\ge 70$. We can find these quartics as follows.

There are 15~quartic relations in
$(\xi_1, \xi_2, \xi_3, \xi_4, \xi_5, \xi_6, \xi_7)$ coming from
the quartics defining~$\CQ$.
They are given by the $4 \times 4$ minors of the matrix~\eqref{E:etamat},
which restricted to~$\CV$ is
\begin{equation} \label{eqn:Matrix}
  M =
   \begin{pmatrix}
     2 f_0 \xi_1 & f_1 \xi_1 & \xi_7 & \xi_6 & \xi_4 \\
     f_1 \xi_1 & 2 (f_2 \xi_1 - \xi_7) & f_3 \xi_1 - \xi_6 & \xi_5 - \xi_4 & \xi_3 \\
     \xi_7 & f_3 \xi_1 - \xi_6 & 2 (f_4 \xi_1 - \xi_5) & f_5 \xi_1 - \xi_3
        & \xi_2 \\
     \xi_6 & \xi_5 - \xi_4 & f_5 \xi_1 - \xi_3 & 2 (f_6 \xi_1 - \xi_2) & f_7 \xi_1 \\
     \xi_4 & \xi_3 & \xi_2 & f_7 \xi_1 & 2 f_8 \xi_1
   \end{pmatrix} .
\end{equation}
Since these relations do not involve~$\xi_8$, they cannot be multiples of
the quadratic relation. We find $55$~further independent~quartics vanishing on~$\CK$
(and thence a basis of the `new' space of quartics that are not
multiples of the quadratic relation) by searching for polynomials
of given degree and weight that vanish on~$\CV$ when pulled back to~$\BA^{15}$.
Removing those that are multiples of the invariant quadric,
we obtain quartics with the following $34$~pairs of degree and weight:
\begin{align*}
  \deg = 4\colon &\quad \wt = 12, 13, 14, 14, 15, 15, 16, 16, 16, 17, 17, 18, 18, 19, 20; \\
  \deg = 5\colon &\quad \wt = 17, 18, 18, 19, 19, 20, 20, 20, 21, 21, 22, 22, 23; \\
  \deg = 6\colon &\quad \wt = 22, 23, 24, 24, 25, 26.
\end{align*}
(Recall that `degree' refers to the degree in terms of the original~$\eta_{ij}$.)
These quartics are given in the file \texttt{Kum3-quartics.magma} at~\cite{Files}.
The quartics are scaled so that their coefficients are in~$\Z[f_0,\ldots,f_8]$.
The 15~quartics of degree~4 are exactly those obtained as $4 \times 4$-minors
of the matrix~$M$ above.

\begin{Lemma} \label{L:quartics_indep}
  Let $f_0, \ldots, f_8 \in k$ be arbitrary. Then the $70$~quartics constructed
  as described above are linearly independent over~$k$.
\end{Lemma}

\begin{proof}
  We can find\star\ $70$ monomials such that the $70 \times 70$-matrix formed by the
  coefficients of the quartics with respect to these monomials has determinant~$\pm 1$.
\end{proof}

Note that regarding~$k$, this is a slight improvement over~\cite[Lemma~3.2]{MuellerG3},
where $k$ was assumed to have characteristic~$\neq 2,3,5$.

We now show that these quartics indeed give all the relations.

\begin{Lemma} \label{L:Mumford_even}
  The natural map $\Sym^2 L(4 \Theta)^+ \to L(8 \Theta)^+$ is surjective.
\end{Lemma}

\begin{proof}
  Mumford shows~\cite[\S4, Thm.~1]{Mumford} that $\Sym^2 L(4 \Theta) \to L(8 \Theta)$
  is surjective. The proof can be modified to give the corresponding result
  for the even subspaces, as follows (we use the notations of~\cite{Mumford}).
  We work with the even functions $\delta_{a+b} + \delta_{-a-b}$ and $\delta_{a-b} + \delta_{-a+b}$.
  This gives
  \begin{align*}
    \sum_{\eta \in Z_2} l(\eta) &(\delta_{a+b+\eta}+\delta_{-a-b-\eta}) * (\delta_{a-b+\eta} + \delta_{-a+b-\eta}) \\
       &= \Bigl(\sum_{\eta \in Z_2} l(\eta) q_1(b+\eta)\Bigr)
           \Bigl(\sum_{\eta \in Z_2} l(\eta) (\delta_{a+\eta} + \delta_{-a-\eta})\Bigr) \\
       &\qquad{} + \Bigl(\sum_{\eta \in Z_2} l(\eta) q_1(a+\eta)\Bigr)
                    \Bigl(\sum_{\eta \in Z_2} l(\eta) (\delta_{b+\eta} + \delta_{-b-\eta})\Bigr) .
  \end{align*}
  We fix the homomorphism $l \colon Z_2 \to \{\pm 1\}$ and the class of $a \bmod K(\delta)$.
  By~(*) in~\cite[p.~339]{Mumford} there is some~$b$ in this class such that
  $\sum_\eta l(\eta) q(b+\eta) \neq 0$. Taking $a = b$, we see that
  \[ \Delta(b) := \sum_\eta l(\eta) (\delta_{b+\eta} + \delta_{-b-\eta}) \]
  is in the image. Using this, we see that for all other~$a$ in the class, $\Delta(a)$
  is also in the image. Inverting the Fourier transform, we find that all $\delta_a + \delta_{-a}$
  are in the image, which therefore consists of all even functions.
\end{proof}

\begin{Corollary} \label{C:sym4}
  The natural map $\Sym^4 L(2 \Theta) \to L(8 \Theta)^+$ is surjective.
\end{Corollary}

\begin{proof}
  Note that $L(2 \Theta) = L(2 \Theta)^+$, so the image of $\Sym^4 L(2 \Theta) \to L(8 \Theta)$
  is contained in the even subspace. Since there is exactly one quadratic relation,
  the map $\Sym^2 L(2 \Theta) \to L(4 \Theta)^+$ is not surjective, but has a
  one-dimensional cokernel. We will see below in Section~\ref{S:dup} that this cokernel is
  generated by the image of a function~$\Xi$ such that $\xi_i \xi_j \Xi$ (for all $i,j$)
  and $\Xi^2$ can be expressed as quartics in the~$\xi_i$. This implies that
  the image of the map in the statement contains the image of $\Sym^2 L(4 \Theta)^+$,
  and surjectivity follows from Lemma~\ref{L:Mumford_even}.
  Note that once we have found $\Xi$ explicitly, the assertions relating to it made above
  can be checked directly and without relying on the considerations leading to the
  determination of~$\Xi$.
\end{proof}

\begin{Theorem} \label{T:Kummerrel}
  Let $k$ be a field of characteristic different from~$2$ and let $F \in k[x,z]$ be
  homogeneous of degree~$8$ and squarefree. Then the image~$\CK$ in~$\BP^7$ of the Kummer variety
  associated to the Jacobian variety of the hyperelliptic curve $y^2 = F(x,1)$
  is defined by the quadric~\eqref{quadrel} and the $34$~quartics constructed above.
\end{Theorem}

\begin{proof}
  By Corollary~\ref{C:sym4} the dimension of the space of quartics vanishing on~$\CK$
  is~$70$. By Lemma~\ref{L:quartics_indep} the quadric and the $34$~quartics give rise
  to $70$~independent quartics vanishing on~$\CK$. By~\cite[Prop.~3.1]{MuellerG3}
  $\CK$ can be defined by quartics, so the claim follows.
\end{proof}

This improves on~\cite[Thm.~3.3]{MuellerG3} by removing the genericity assumption
(and allowing characteristic $3$ or~$5$).

To conclude this section, we determine the images of some special points on~$\CJ$
under the map to~$\CK$.

The discussion on page~\pageref{asymp} shows that
on a point $[(x_1,y_1)+(x_2,y_2)+\fm] \in \Theta$, the map restricts to
\begin{align*} \Bigl(0 &: 1 : -(x_1+x_2) : x_1 x_2 : x_1^2 + x_1 x_2 + x_2^2 \\
                       &: -(x_1+x_2) x_1 x_2 : (x_1 x_2)^2
                        : \frac{2 y_1 y_2 - G(x_1,x_2)}{(x_1-x_2)^2}\Bigr) .
\end{align*}
If we write $(X - x_1)(X - x_2) = \sigma_0 X^2 + \sigma_1 X + \sigma_2$, then
this can be written as
\[ (0 : \sigma_0^2 : \sigma_0 \sigma_1 : \sigma_0 \sigma_2
      : \sigma_1^2 - \sigma_0 \sigma_2 : \sigma_1 \sigma_2 : \sigma_2^2
      : \xi_8) ,
\]
where, rewriting $\bigl((x_1 - x_2)^2 \xi_8 - G(x_1, x_2)\bigr)^2 = 4 F(x_1, 1) F(x_2, 1)$,
we have that
\begin{align*}
 (\sigma_1^2 - &4 \sigma_0 \sigma_2) \xi_8^2 \\
   &{}
      + (4 f_0 \sigma_0^4 - 2 f_1 \sigma_0^3 \sigma_1 + 4 f_2 \sigma_0^3 \sigma_2
      - 2 f_3 \sigma_0^2 \sigma_1 \sigma_2 + 4 f_4 \sigma_0^2 \sigma_2^2 \\
   &\hphantom{{} + (}
      - 2 f_5 \sigma_0 \sigma_1 \sigma_2^2 + 4 f_6 \sigma_0 \sigma_2^3
      - 2 f_7 \sigma_1 \sigma_2^3 + 4 f_8 \sigma_2^4) \xi_8 \\
   &{}
      + (-4 f_0 f_2 + f_1^2) \sigma_0^6 + 4 f_0 f_3 \sigma_0^5 \sigma_1
      - 2 f_1 f_3 \sigma_0^5 \sigma_2 - 4 f_0 f_4 \sigma_0^4 \sigma_1^2 \\
   &\hphantom{{} + (}
      + (-4 f_0 f_5 + 4 f_1 f_4) \sigma_0^4 \sigma_1 \sigma_2
      + (-4 f_0 f_6 + 2 f_1 f_5 - 4 f_2 f_4 + f_3^2) \sigma_0^4 \sigma_2^2 \\
   &\hphantom{{} + (}
      + 4 f_0 f_5 \sigma_0^3 \sigma_1^3
      + (8 f_0 f_6 - 4 f_1 f_5) \sigma_0^3 \sigma_1^2 \sigma_2
      + (8 f_0 f_7 - 4 f_1 f_6 + 4 f_2 f_5) \sigma_0^3 \sigma_1 \sigma_2^2 \\
   &\hphantom{{} + (}
      + (-2 f_1 f_7 - 2 f_3 f_5) \sigma_0^3 \sigma_2^3
      - 4 f_0 f_6 \sigma_0^2 \sigma_1^4
      + (-12 f_0 f_7 + 4 f_1 f_6) \sigma_0^2 \sigma_1^3 \sigma_2 \\
   &\hphantom{{} + (}
      + (-16 f_0 f_8 + 8 f_1 f_7 - 4 f_2 f_6) \sigma_0^2 \sigma_1^2 \sigma_2^2
      + (8 f_1 f_8 - 4 f_2 f_7 + 4 f_3 f_6) \sigma_0^2 \sigma_1 \sigma_2^3 \\
   &\hphantom{{} + (}
      + (-4 f_2 f_8 + 2 f_3 f_7 - 4 f_4 f_6 + f_5^2) \sigma_0^2 \sigma_2^4
      + 4 f_0 f_7 \sigma_0 \sigma_1^5 \\
   &\hphantom{{} + (}
      + (16 f_0 f_8 - 4 f_1 f_7) \sigma_0 \sigma_1^4 \sigma_2
      + (-12 f_1 f_8 + 4 f_2 f_7) \sigma_0 \sigma_1^3 \sigma_2^2 \\
   &\hphantom{{} + (}
      + (8 f_2 f_8 - 4 f_3 f_7) \sigma_0 \sigma_1^2 \sigma_2^3
      + (-4 f_3 f_8 + 4 f_4 f_7) \sigma_0 \sigma_1 \sigma_2^4
      - 2 f_5 f_7 \sigma_0 \sigma_2^5 \\
   &\hphantom{{} + (}
      - 4 f_0 f_8 \sigma_1^6
      + 4 f_1 f_8 \sigma_1^5 \sigma_2 - 4 f_2 f_8 \sigma_1^4 \sigma_2^2
      + 4 f_3 f_8 \sigma_1^3 \sigma_2^3 - 4 f_4 f_8 \sigma_1^2 \sigma_2^4 \\
   &\hphantom{{} + (}
      + 4 f_5 f_8 \sigma_1 \sigma_2^5 + (-4 f_6 f_8 + f_7^2) \sigma_2^6 \\
   &= 0 .
\end{align*}
(This is similar to the quartic defining the Kummer surface in the genus~$2$ case.)
The image on~$\CK$ of the theta divisor is a surface of degree~$12$ in
$\BP^6 = \BP^7 \cap \{\xi_1 = 0\}$; the intersection of~$\CK$ with the
hyperplane $\xi_1 = 0$ is twice the image of~$\Theta$. (The equation above
is cubic in the middle six coordinates and~$\xi_8$, so we get three times
the degree of the Veronese surface. It is known that $\CK$ has degree~$24$.)

When $(x_2,y_2)$ approaches $(x_1,-y_1)$, then the last coordinate tends to infinity,
whereas the remaining ones stay bounded, so the origin on~$\CJ$ is mapped to
\[ o := (0 : 0 : 0 : 0 : 0 : 0 : 0 : 1) . \]

Points in~$\CJ[2]$ are represented by factorizations $F = G H$ with $d = \deg G$
even, compare Section~\ref{S:2tors} below. Writing
\[ G = g_d x^d + g_{d-1} x^{d-1} z + \ldots + g_0 z^d \quad\text{and}\quad
   H = h_{8-d} x^{8-d} + h_{7-d} x^{7-d} z + \ldots + h_0 z^{8-d} ,
\]
we see that a 2-torsion point represented by $(G, H)$ with $\deg G = 2$ maps to
\begin{equation} \label{E:tors1}
  (0 : g_2^2 : g_1 g_2 : g_0 g_2
      : g_1^2 - g_0 g_2 : g_0 g_1 : g_0^2
      : g_0^3 h_6 + g_0^2 g_2 h_4 + g_0 g_2^2 h_2 + g_2^3 h_0) .
\end{equation}
A 2-torsion point represented by $(G, H)$ with $\deg G = 4$ maps to
\begin{align}
  \bigl(1 &: g_2 h_4 + g_4 h_2 : g_1 h_4 + g_4 h_1 : g_0 h_4 + g_4 h_0 \label{E:tors2} \\
          &: g_0 h_4 + g_4 h_0 + g_1 h_3 + g_3 h_1 \nonumber
           : g_0 h_3 + g_3 h_0 : g_0 h_2 + g_2 h_0 \\
          &: (g_0 h_4 + g_4 h_0)^2
              + (g_0 h_2 + g_2 h_0)(g_2 h_4 + g_4 h_2)
              + (g_1 h_0 - g_0 h_1)(g_4 h_3 - g_3 h_4)
  \bigr) ; \nonumber
\end{align}
this is obtained by taking $(A,B,C) = (G,0,H)$ in our original
parameterization.


\section{Transformations} \label{S:trans}

We compare our coordinates for the Kummer variety with those
of Stubbs~\cite{Stubbs}, Duquesne~\cite{Duquesne} and M\"uller~\cite{MuellerThesis}
in the special case $f_8 = 0$. In this case there is a rational Weierstrass
point at infinity, and we can fix the representation of a point outside of~$\Theta$
by requiring that $A$ vanishes at infinity and that $\deg B(x,1) < \deg A(x,1)$.
For a generic point~$P$ on~$\CJ$, $\deg A(x,1) = 3$; let $(x_j,y_j)$ for $j = 1,2,3$
be the three points in the effective divisor~$D$ such that $P = [D - 3 \cdot \infty]$.
Generically, the three points are distinct. Then
\[ A(x,1) = (x - x_1) (x - x_2) (x - x_3) \]
and $B(x,1)$ is the interpolation polynomial such that $B(x_j, 1) = y_j$ for $j = 1,2,3$.
We obtain the~$c_j$ from $C = (B^2 - F)/A$ by polynomial division.
This leads to\star
\[ \begin{array}{r@{{}={}}rc@{}rcrcrcrcrcrcr}
     \xi_1 & \kappa_1 \\
     \xi_2 &  & & -f_7 \kappa_2 \\
     \xi_3 &  & &  & & f_7 \kappa_3 \\
     \xi_4 &  & &  & &  & & -f_7 \kappa_4 \\
     \xi_5 & f_4 \kappa_1 &+& f_5 \kappa_2 &+& 2 f_6 \kappa_3 &+& 3 f_7 \kappa_4
            &-& \kappa_5 \\
     \xi_6 & f_3 \kappa_1 &+& f_4 \kappa_2 &+& f_5 \kappa_3 & &
            & &  &-& \kappa_6 \\
     \xi_7 & f_2 \kappa_1 & &  &-& f_4 \kappa_3 &-& 3 f_5 \kappa_4
            & &  & &  &-& \kappa_7 \\
     \xi_8 &  &-& f_2 f_7 \kappa_2 &-& f_3 f_7 \kappa_3 &-& f_4 f_7 \kappa_4
            & &  & &  & &  &+& f_7 \kappa_8
   \end{array}
\]
where $\kappa_1, \kappa_2, \ldots, \kappa_8$ are the coordinates used
by the other authors.

We consider the effect of a transformation of the curve equation.
First suppose that $\tilde{F}(x,z) = F(x + \lambda z, z)$ (corresponding
to a shift of the $x$-coordinate in the affine equation). A point represented
by a triple $(A(x,z), B(x,z), C(x,z))$ of polynomials will correspond to
the point $(\tilde{A}(x,z), \tilde{B}(x,z), \tilde{C}(x,z))$ with
$\tilde{A}(x,z) = A(x + \lambda z, z)$ and analogously for $\tilde{B}$
and~$\tilde{C}$. We obtain\star
\begin{align*}
  \tilde{\xi}_1
   &= \xi_1 \\
  \tilde{\xi}_2
    &= \xi_2 + 3 \lambda f_7 \xi_1 + 12 \lambda^2 f_8 \xi_1 \\
  \tilde{\xi}_3
    &= \xi_3 + 2 \lambda \xi_2 + 3 \lambda^2 f_7 \xi_1 + 8 \lambda^3 f_8 \xi_1 \\
  \tilde{\xi}_4
    &= \xi_4 + \lambda \xi_3 + \lambda^2 \xi_2
        + \lambda^3 f_7 \xi_1 + 2 \lambda^4 f_8 \xi_1 \\
  \tilde{\xi}_5
    &= \xi_5
        + \lambda (2 f_5 \xi_1 + 3 \xi_3)
        + \lambda^2 (6 f_6 \xi_1 + 3 \xi_2)
        + 17 \lambda^3 f_7 \xi_1
        + 34 \lambda^4 f_8 \xi_1 \\
  \tilde{\xi}_6
    &= \xi_6
        + \lambda (3 \xi_4 + \xi_5)
        + \lambda^2 (f_5 \xi_1 + 3 \xi_3)
        + \lambda^3 (2 f_6 \xi_1 + 2 \xi_2)
        + 5 \lambda^4 f_7 \xi_1
        + 8 \lambda^5 f_8 \xi_1 \\
  \tilde{\xi}_7
   &= \xi_7
       + \lambda (f_3 \xi_1 + 2 \xi_6)
       + \lambda^2 (2 f_4 \xi_1 + 3 \xi_4 + \xi_5)
       + \lambda^3 (4 f_5 \xi_1 + 2 \xi_3) \\
   &\qquad{}
       + \lambda^4 (6 f_6 \xi_1 + \xi_2)
       + 9 \lambda^5 f_7 \xi_1
       + 12 \lambda^6 f_8 \xi_1 \\
  \tilde{\xi}_8
    &= \xi_8 + \lambda
       (f_3 \xi_2 + 2 f_5 \xi_4 + 3 f_7 \xi_7) \\
    &\qquad{} + \lambda^2
       (3 f_3 f_7 \xi_1 + 2 f_4 \xi_2  + f_5 \xi_3 + 6 f_6 \xi_4 + 3 f_7 \xi_6
         + 12 f_8 \xi_7) \\
    &\qquad{} + \lambda^3
       ((12 f_3 f_8 + 6 f_4 f_7) \xi_1 + 4 f_5 \xi_2 + 4 f_6 \xi_3 + 17 f_7 \xi_4
         + f_7 \xi_5 + 16 f_8 \xi_6) \\
    &\qquad{} + \lambda^4
       ((24 f_4 f_8 + 11 f_5 f_7) \xi_1 + 8 f_6 \xi_2 + 12 f_7 \xi_3 + 46 f_8 \xi_4
         + 6 f_8 \xi_5) \\
    &\qquad{} + \lambda^5
       ((44 f_5 f_8 + 18 f_6 f_7) \xi_1 + 16 f_7 \xi_2 + 32 f_8 \xi_3) \\
    &\qquad{} + \lambda^6
       ((68 f_6 f_8 + 29 f_7^2) \xi_1 + 32 f_8 \xi_2)
        + 148 \lambda^7 f_7 f_8 \xi_1 + 148 \lambda^8 f_8^2 \xi_1.
\end{align*}

For the transformation given by $\tilde{F}(x,z) = F(z,x)$, we have
\[ \tilde{a}_j = a_{4-j}, \qquad \tilde{b}_j = b_{4-j}, \qquad \tilde{c}_j = c_{4-j} \]
and therefore
\[ (\tilde{\xi}_1, \tilde{\xi}_2, \tilde{\xi}_3, \tilde{\xi}_4,
    \tilde{\xi}_5, \tilde{\xi}_6, \tilde{\xi}_7, \tilde{\xi}_8)
    = (\xi_1, \xi_7, \xi_6, \xi_4, \xi_5, \xi_3, \xi_2, \xi_8) .
\]

More generally, consider an element
\[ \sigma = \begin{pmatrix} r & s \\ t & u \end{pmatrix} \in \GL(2) \]
acting by $(x,z) \mapsto (rx+sz, tx+uz)$. Let $\Sigma \in \GL(5)$ be the
matrix whose columns are the coefficients of $(rx+sz)^j (tx+uz)^{4-j}$,
for $j = 0,1,2,3,4$ (this is the matrix giving the action of~$\sigma$
on the fourth symmetric power of the standard representation of~$\GL(2)$).
Recall the matrix~$L$ from~\eqref{eqn:L} whose rows contain the coefficients
of $A$, $B$ and~$C$.
Then the effect on our variables $a_i$, $b_i$, $c_i$ is given by
$L \mapsto L \Sigma^\top$. With $D$ as in~\eqref{eqn:D},
we have $L^\top D L = M$ with $M$ as in~\eqref{eqn:Matrix}. So the effect of~$\sigma$ on~$M$
is given by $M \mapsto \Sigma M \Sigma^\top$. Note that $\tilde{\xi}_1 = \xi_1$
and that we can extract $\tilde{\xi}_2, \ldots, \tilde{\xi}_7$ from~$M$;
to get $\tilde{\xi}_8$ when $\xi_1$ is not invertible, we can perform a generic
computation and then specialize.

This allows us to reduce our more general setting to the situation when
there is a Weierstrass point at infinity: we adjoin a root of~$F(x,1)$,
then we shift this root to zero and invert. This leads to an equation
with $f_8 = 0$. This was used to obtain the matrix representing the
action of an even $2$-torsion point, see below in Section~\ref{S:2tors}.


\section{Lifting points to the Jacobian} \label{S:lift}

Let $P \in \CK(k)$ be a $k$-rational point on the Kummer variety.
We want to decide if $P = \kappa(P')$ for a $k$-rational point~$P'$ on the Jacobian~$\CJ$.
Consider an odd function~$h$ on~$\CJ$ (i.e., such that $h(-Q) = -h(Q)$ for $Q \in \CJ$)
such that $h$ is defined over~$k$; then $h(P') \in k$ (or $h$ as a pole at~$P'$).
Since $h^2$ is an even function, it descends to a function on~$\CK$, and we must have that
$h^2(P) = h^2(P') = h(P')^2$ is a square in~$k$. Conversely, any non-zero
odd function~$h$ on~$\CJ$ will generically separate the two points in the
fiber of the double cover $\CJ \to \CK$, so if $h^2(P)$ is a non-zero square in~$k$,
then this implies that $P$ lifts to a $k$-rational point on~$\CJ$.

So we will now exhibit some odd functions that we can use to decide if a point lifts.
Since $L(2\Theta)$ consists of even functions only, we look at
$L(3\Theta)$, which has dimension $3^3 = 27$.
Its subspace of even functions has dimension~$14$ and is spanned by
$\xi_1, \ldots, \xi_8$, the five quadratics
\[ \xi_2 (\xi_4 + \xi_5) - \xi_3^2, \quad
   \xi_2 \xi_6 - \xi_3 \xi_4, \quad
   \xi_2 \xi_6 - \xi_4^2, \quad
   \xi_3 \xi_6 - \xi_4 \xi_7, \quad
   (\xi_4 + \xi_5) \xi_7 - \xi_6^2
\]
and a further function, which can be taken to be\star
\begin{align*}
  2(2 f_0 \xi_2^2 &- f_1 \xi_2 \xi_3 + 2 f_2 \xi_2 \xi_4 - f_3 \xi_2 \xi_6
    + 2 f_4 \xi_2 \xi_7 - f_5 \xi_3 \xi_7 + 2 f_6 \xi_4 \xi_7 - f_7 \xi_6 \xi_7
    + 2 f_8 \xi_7^2) \\
    &{}- 7 \xi_2 \xi_4 \xi_7 + \xi_2 \xi_5 \xi_7 + \xi_2 \xi_6^2 + \xi_3^2 \xi_7
    + 4 \xi_3 \xi_4 \xi_6 - 2 \xi_3 \xi_5 \xi_6 + \xi_4^3 - 5 \xi_4^2 \xi_5
    + 2 \xi_4 \xi_5^2 .
\end{align*}
The subspace of odd functions has
dimension~$13$. We obtain a ten-dimensional subspace of this space by
considering the coefficients of $A_l \wedge B_l \wedge C_l$, which is
an expression of degree~$3$, of odd degree in~$B$ and invariant even under
$\SL(3)$ acting on $(A, B, C)$. (One can check that there are no further
$\Gamma$-invariants of degree~$3$.) These coefficients are
given by the $3 \times 3$-minors of the matrix~$L$ of~\eqref{eqn:L}.
If we denote the minor corresponding to $0 \le i < j < k \le 4$ by
$\mu_{ijk}$, then we find that
\begin{equation} \label{E:muijk}
  \mu_{ijk}^2 = \eta_{ii} \eta_{jk}^2 + \eta_{jj} \eta_{ik}^2 + \eta_{kk} \eta_{ij}^2
                  - 4 \eta_{ii} \eta_{jj} \eta_{kk} - \eta_{ij} \eta_{ik} \eta_{jk} .
\end{equation}
If $L_{ijk}$ is the corresponding $3 \times 3$ submatrix of~$L$, then we have that
\[ \mu_{ijk}^2 = \det(L_{ijk})^2 = -\tfrac{1}{2} \det (L_{ijk}^\top D L_{ijk}) \]
with~$D$ as in~\eqref{eqn:D}.
We also have that $L^\top D L = M$, where
$M$ is the matrix corresponding to the quadratic form $B_l^2 - A_l C_l$
given in~\eqref{eqn:Matrix}. We can express this by saying
that $\mu_{ijk}^2$ is $-\frac{1}{2}$ times the corresponding principal minor of~$M$.
In the same way, one sees that $\mu_{ijk} \mu_{i'j'k'}$ is $-\frac{1}{2}$ times
the minor of~$M$ given by selecting rows $i,j,k$ and columns $i',j',k'$.
This shows that if one $\mu_{ijk}^2(P)$ is a non-zero square in~$k$, then
all $\mu_{i'j'k'}^2(P)$ are squares in~$k$. All ten of them
vanish simultaneously if and only if $A$, $B$ and~$C$ are linearly dependent
(this is equivalent to the rank of $B_l^2 - A_l C_l$ being at most~$2$).
The dimension of the space spanned by $A$, $B$ and~$C$ cannot be strictly less than~$2$,
since this would imply that $F$ is a constant
times a square, which contradicts the assumption that $F$ is squarefree.
So we can write $A$, $B$ and~$C$ as linear combinations of two~polynomials
$A'$ and~$C'$, and after a suitable change of basis, we find that
$F = B^2 - A C = A' C'$. This means that the point is the image of
a $2$-torsion point on~$\CJ$, and it will always lift.

So for a point~$P$ in~$\CK(k)$ with $\xi_1 = 1$ (hence outside the theta divisor)
to lift to a point in~$\CJ(k)$, it is necessary that all these expressions,
when evaluated at~$P$, are squares in~$k$, and sufficient that one of them gives
a non-zero square.
For points with $\xi_1 = 0$, we can use the explicit description of the
image of~$\Theta$ given in Section~\ref{S:coords}.

Let $\CV'$ be the quotient of~$\CV$ by the action of the subgroup of~$\Gamma$
generated by the elements of the form $t_\lambda$ and~$n_\mu$; then the points
of~$\CV'$ correspond to effective divisors of degree~$4$ on~$\CC$ in general position.
Geometrically, the induced map $\CV' \to \CX \setminus \Theta$ is a conic bundle:
for a point on~$\CX$ outside the theta divisor, all effective divisors
representing it are in general position, and the corresponding linear
system has dimension~$1$ by the Riemann-Roch Theorem, so the fibers are
Severi-Brauer varieties of dimension~$1$. If $\CC$ has a $k$-rational point~$P$,
then the bundle has a section (and so is in fact a $\BP^1$-bundle), since
we can select the unique representative containing~$P$ in its support.
If $k$ is a number field and $\CC$ has points over every completion of~$k$,
then all the conics in fibers above $k$-rational points on~$\CX \setminus \Theta$
have points over all completions of~$k$ and therefore are isomorphic to~$\BP^1$
over~$k$. We can check whether a $k$-defined divisor representing a lift
of~$P$ to a $k$-rational point on~$\CJ$ exists and find one in this case
in the following way. We assume that $P$ is not in the image of~$\Theta$
and is not the image of a $2$-torsion point.
We are looking for a matrix $\tilde{L} \in \BA^{15}(k)$ representing
a lift $P' \in \CJ(k)$ of~$P$. Since we exclude $2$-torsion,
the matrix~$\tilde{L}$ must have rank~$3$, and there is a minor~$\mu_{ijk}$
such that $\mu_{ijk}^2(P) = \mu_{ijk}(P')^2$ is a non-zero square in~$k$.
The rank of~$M(P) = \tilde{L}^\top D \tilde{L}$ is also~$3$, so
both $L(\tilde{P})$ and~$M(P)$ have the same $2$-dimensional kernel.
We can compute the kernel from~$M(P)$ and then we find the space generated
by the rows of~$\tilde{L}$ as its annihilator, which is simply
given by rows $i,j,k$ of~$M(P)$. If we find an invertible $3 \times 3$ matrix~$U$
with entries in~$k$
such that $M_{ijk}(P) = U^\top D U$ (where $M_{ijk}$ is the principal $3 \times 3$
submatrix of~$M$ given by rows and columns $i,j,k$), then we can find a
suitable matrix~$\tilde{L}$
whose rows are in the space generated by rows $i,j,k$ of~$M(P)$ and such that
$\tilde{L}_{ijk} = U$. Then $\tilde{L}^\top D \tilde{L} = M(P)$, so $\tilde{L}$
gives us the desired representative. Finding~$U$ is equivalent
to finding an isomorphism between the quadratic forms given by
\[ (x_1, x_2, x_3) M_{ijk}(P) (x_1, x_2, x_3)^\top  \qquad\text{and}\qquad
   2 x_1 x_3 - 2 x_2^2 ,
\]
for whose existence a necessary condition is that $\det M_{ijk}(P) = -2 \mu_{ijk}^2(P)$
is a square times $\det D = -2$. Given this, the problem
comes down to finding a point on the conic given by the first form
(which is the conic making up the fiber above~$P'$ or~$-P'$) and then parameterizing
the conic using lines through the point.

\begin{Remark}
  One can check\star\ that the following three expressions are a possible choice
  for the missing three basis elements of the odd subspace of~$L(3\Theta)$:
  \begin{align*}
     &\xi_2 \mu_{012} - \xi_3 \mu_{013} + \xi_5 \mu_{014} \\
     &\xi_3 \mu_{014} - (\xi_4 + \xi_5) \mu_{024} + \xi_4 \mu_{123} + \xi_6 \mu_{034} \\
     &\xi_5 \mu_{034} - \xi_6 \mu_{134} + \xi_7 \mu_{234}
  \end{align*}
\end{Remark}


\section{The action of the $2$-torsion subgroup on~$\CK$} \label{S:2tors}

We follow the approach taken in~\cite{StollH1} and consider the action of
the $2$-torsion subgroup of~$\CJ$ on~$\CK$ and the ambient projective space.
Note that translation by a $2$-torsion point commutes with negation on~$\CJ$,
so the translation descends to an automorphism of~$\CK$, and since $2\Theta$
is linearly equivalent to its translate, this automorphism actually is induced
by an automorphism of the ambient~$\BP^7$.

We will see that this projective representation of~$\CJ[2] \simeq (\Z/2\Z)^6$
can be lifted to a representation of a central extension of~$\CJ[2]$ by~$\mu_2$
on the space of linear forms in the coordinates $\xi_1, \ldots, \xi_8$.
This representation is irreducible. In the next section, we consider this
representation and the induced representations on the spaces of quadratic
and quartic forms in $\xi_1, \ldots, \xi_8$, whereas in this section, we
obtain an explicit description of the action of~$\CJ[2]$ on~$\BP^7$.

There is a natural bijection between the $2$-torsion subgroup~$\CJ[2]$
of the Jacobian and the set of unordered partitions of the set
$\Omega \subset \BP^1$ of zeros of~$F$ into two subsets of even cardinality.
The torsion point~$T$ corresponding to a partition $\{\Omega_1, \Omega_2\}$ is
\[ \left[\sum_{\omega \in \Omega_1} (\omega,0)\right] - \frac{\#\Omega_1}{2} \fm
      = \left[\sum_{\omega \in \Omega_2} (\omega,0)\right] - \frac{\#\Omega_2}{2} \fm .
\]
Since $\#\Omega = 8$ is divisible by~$4$, the quantity
$\eps(T) = (-1)^{\#\Omega_1/2} = (-1)^{\#\Omega_2/2}$ is well-defined.
We say that $T$ is \emph{even} if $\eps(T) = 1$ and \emph{odd} if $\eps(T) = -1$.
By definition, the even $2$-torsion points are the 35~points
corresponding to a partition into two sets of four roots, together with the origin,
and the odd
$2$-torsion points are the 28~points corresponding to a partition into
subsets of sizes $2$ and~$6$. The Weil pairing of two torsion points
$T$ and~$T'$ represented by $\{\Omega_1, \Omega_2\}$ and~$\{\Omega'_1, \Omega'_2\}$,
respectively, is given by
\[ e_2(T, T') = (-1)^{\#(\Omega_1 \cap \Omega'_1)} . \]
It is then easy to check that
\begin{equation} \label{eqn:Weil}
  e_2(T, T') = \eps(T) \eps(T') \eps(T+T') .
\end{equation}
Note that $\Pic^0_\CC$ is canonically isomorphic to~$\Pic^2_\CC$ (by adding the
class of~$\fm$), which contains the theta characteristics.
(A divisor class $\fD \in \Pic^2_\CC$ is
a \emph{theta characteristic} if $2\fD = \fW$.) In this way,
the theta characteristics are identified with the $2$-torsion points,
and the odd (resp., even) theta characteristics correspond to the
odd (resp., even) $2$-torsion points.

Using the transformations described in Section~\ref{S:trans} and the matrices obtained by
Duquesne~\cite{Duquesne} representing the translation by a $2$-torsion point,
we find the corresponding matrices in our setting for an even nontrivial
$2$-torsion point. The matrices corresponding to odd $2$-torsion points can then also
be derived. For each factorization $F = G H$ into two forms of even degree,
there is a matrix $M_{(G,H)}$ whose entries are polynomials with integral coefficients
in the coefficients of $G$ and~$H$ and whose image in $\PGL(8)$ gives the action
of the corresponding $2$-torsion point.
These entries are too large to be reproduced here, but
are given in the file \texttt{Kum3-torsionmats.magma} at~\cite{Files}.

The matrices satisfy the relations\star
\begin{equation} \label{eqn:MT}
  M_{(G,H)}^2 = \Res(G, H) I_8 \qquad\text{and}\qquad \det M_{(G,H)} = \Res(G, H)^4 ,
\end{equation}
where $\Res$ denotes the resultant of two binary forms. Let
\[ S = \left(
       \begin{array}{@{\,}r@{\;}r@{\;\;}r@{\;}r@{\;}r@{\;\;}r@{\;}r@{\;\;}r@{\,}}
         0 & 0 & 0 & 0 &  0 & 0 & 0 & 1 \\
         0 & 0 & 0 & 0 &  0 & 0 &-1 & 0 \\
         0 & 0 & 0 & 0 &  0 & 1 & 0 & 0 \\
         0 & 0 & 0 & 0 & -1 & 0 & 0 & 0 \\
         0 & 0 & 0 &-1 &  0 & 0 & 0 & 0 \\
         0 & 0 & 1 & 0 &  0 & 0 & 0 & 0 \\
         0 &-1 & 0 & 0 &  0 & 0 & 0 & 0 \\
         1 & 0 & 0 & 0 &  0 & 0 & 0 & 0
       \end{array}\right)
\]
be the matrix corresponding to the quadratic
relation~\eqref{quadrel} satisfied by points on the Kummer variety.

\begin{Definition}
  We will write $\langle \cdot, \cdot \rangle_S$ for the pairing given by~$S$.
  Concretely, for vectors $\vxi = (\xi_1, \ldots, \xi_8)$ and
  $\vz = (\zeta_1, \ldots, \zeta_8)$, we have
  \[ \langle \vxi, \vz \rangle_S
      = \xi_1 \zeta_8 - \xi_2 \zeta_7 + \xi_3 \zeta_6 - \xi_4 \zeta_5
        - \xi_5 \zeta_4 + \xi_6 \zeta_3 - \xi_7 \zeta_2 + \xi_8 \zeta_1 .
  \]
\end{Definition}

One checks\star\ that for all $G, H$ as above,
\[ (S M_{(G,H)})^\top = (-1)^{(\deg G)/2} S M_{(G,H)} . \]
If $T \neq 0$ is even, then all corresponding matrices $M_{(G,H)}$ are equal;
we denote this matrix by~$M_T$. In this case, also the resultant $\Res(G, H)$
depends only on~$T$; we write it~$r(T)$, so that we have $M_T^2 = r(T) I_8$.
For $T$ odd and represented by $(G, H)$ with $\deg G = 2$, we have
$M_{(\lambda G, \lambda^{-1} H)} = \lambda^2 M_{(G,H)}$.
As a special case, we have $M_{(1, F)} = I_8$.
For $T \neq 0$ even, the entry in the upper right corner of~$M_T$ is~$1$, for
all other $2$-torsion points, this entry is zero.

For a $2$-torsion point $T \in \CJ[2]$, if we denote by $M_T$ the matrix corresponding
to one of the factorizations defining~$T$, we therefore have
(using that $S = S^\top = S^{-1}$)
\[ (S M_T)^\top = \eps(T) S M_T, \qquad\text{or equivalently,}\quad
    M_T = \eps(T) S M_T^\top S .
\]
This implies (using that $M_{T'} M_T$ is, up to scaling, a matrix corresponding
to $T+T'$)
\begin{align*}
  M_T M_{T'} &= \eps(T) S M_T^\top S \cdot \eps(T') S M_{T'}^\top S \\
             &= \eps(T) \eps(T') S (M_{T'} M_T)^\top S
              = \eps(T) \eps(T') \eps(T+T') M_{T'} M_T .
\end{align*}
Using~\eqref{eqn:Weil}, we recover the well-known fact that
\begin{equation} \label{eqn:e2pairing}
  M_T M_{T'} = e_2(T, T') M_{T'} M_T .
\end{equation}

Since $M_T^2$ is a scalar matrix, the relation given above implies that
the quadratic relation is invariant (up to scaling) under the action
of~$\CJ[2]$ on~$\BP^7$:
\[ M_T^\top S M_T = \Res(G,H) S . \]


\section{The action on linear, quadratic and quartic forms} \label{S:action}

We work over an algebraically closed field~$k$ of characteristic different
from~$2$. The first result describes a representation of a central extension~$G$
of~$\CJ[2]$ on the space of linear forms that lifts the action on~$\BP^7$.

\begin{Lemma} \label{L:J2lift}
  There is a subgroup~$G$ of~$\SL(8)$ and an exact sequence
  \[ 0 \To \mu_2 \To G \To \CJ[2] \To 0 \]
  induced by the standard sequence
  \[ 0 \To \mu_8 \To \SL(8) \To \PSL(8) \To 0 \]
  and the embedding $\CJ[2] \to \PSL(8)$ given by associating to~$T$
  the class of any matrix~$M_T$.
\end{Lemma}

\begin{proof}
  Let $T \in \CJ[2]$ and let $M_T \in \GL(8)$ be any matrix associated
  to~$T$. Then $M_T^2 = c I_8$ with some~$c$ (compare~\eqref{eqn:MT}),
  and we let $\tilde{M}_T$ denote one of the two matrices $\gamma M_T$
  where $\gamma^2 c = \eps(T)$. Then $\tilde{M}_T \in \SL(8)$, since
  (again by~\eqref{eqn:MT})
  \[ \det \tilde{M}_T = \gamma^8 \det M_T = (\eps(T) c^{-1})^4 c^4 = 1 . \]
  Since any two choices of~$M_T$ differ only by scaling, $\tilde{M}_T$
  is well-defined up to sign. Among the lifts of the class of~$M_T$ in~$\PSL(8)$
  to~$\SL(8)$, $\pm \tilde{M}_T$ are characterized by the relation
  $\tilde{M}_T^2 = \eps(T) I_8$. We now set
  \[ G = \{\pm \tilde{M}_T : T \in \CJ[2]\} . \]
  It is clear
  that $G$ surjects onto the image of~$\CJ[2]$ in~$\PSL(8)$ and that the
  map is two-to-one. It remains to show that $G$ is a
  group. So let $T, T' \in \CJ[2]$. Then $\tilde{M}_T \tilde{M}_{T'}$
  is a matrix corresponding to~$T+T'$.
  Since (using \eqref{eqn:e2pairing} and~\eqref{eqn:Weil})
  \begin{align*}
    (\tilde{M}_T \tilde{M}_{T'})^2
      &= \tilde{M}_T \tilde{M}_{T'} \tilde{M}_T \tilde{M}_{T'}
       = e_2(T, T') \tilde{M}_T^2 \tilde{M}_{T'}^2 \\
      &= e_2(T, T') \eps(T) \eps(T') I_8
       = \eps(T+T') I_8 ,
  \end{align*}
  we find that $\tilde{M}_T \tilde{M}_{T'} \in G$.
\end{proof}

\begin{Remark}
  Note that the situation here is somewhat different from the situation
  in genus~$2$, as discussed in~\cite{StollH1}. In the even genus hyperelliptic
  case, the theta characteristics live in~$\Pic^{\text{odd}}$ rather than
  in~$\Pic^{\text{even}}$ and can therefore not be identified with the
  $2$-torsion points. The effect is that there is no map $\eps \colon \CJ[2] \to \mu_2$
  that induces the Weil pairing as in~\eqref{eqn:Weil}, so that we have to
  use a fourfold covering of~$\CJ[2]$ in~$\SL(4)$ rather than a double cover.
\end{Remark}

We now proceed to a study of the representations of~$G$ on linear, quadratic
and quartic forms on~$\BP^7$ that are induced by $G \subset \SL(8)$.
The representation~$\rho_1$ on the space~$V_1$ of linear forms is the standard
representation. For its character~$\chi_1$, we find that
\[ \chi_1(\pm I_8) = \pm 8 \qquad\text{and}\qquad \chi_1(\pm \tilde{M}_T) = 0
   \quad\text{for all $T \neq 0$.}
\]
This follows from the observation that $T$ can be written as $T = T' + T''$
with $e_2(T', T'') = -1$. Since
$\pm \tilde{M}_T = \tilde{M}_{T'} \tilde{M}_{T''} = -\tilde{M}_{T''} \tilde{M}_{T'}$,
the trace of~$\tilde{M}_T$ must be zero. We deduce that~$\rho_1$ is irreducible.
($\rho_1$ is essentially the representation~$V(\delta)$ in~\cite{Mumford},
where $\delta = (2,2,2)$ in our case.)

The representation~$\rho_2$ on the space~$V_2$ of quadratic forms
is the symmetric square of~$\rho_1$.
Since $\pm I_8$ act trivially on even degree forms, $\rho_2$ descends to a
representation of~$\CJ[2]$. Its character~$\chi_2$ is given by
\begin{align*}
  \chi_2(0) &= 36 \qquad\text{and}\qquad \\
  \chi_2(T) &= \tfrac{1}{2}\bigl(\chi_1(\tilde{M}_T)^2 + \chi_1(\tilde{M}_T^2)\bigr)
             = \tfrac{1}{2}(0 + 8 \eps(T))
             = 4 \eps(T) \quad \text{for $T \neq 0$.}
\end{align*}
Since $\CJ[2]$ is abelian, this representation has to split into a direct sum
of one-dimensional representations. Define the character~$\chi_T$ of~$\CJ[2]$
by $\chi_T(T') = e_2(T, T')$. Then the above implies that
\begin{equation} \label{eqn:rho2}
  \rho_2 = \bigoplus_{T \colon \eps(T) = 1} \chi_T .
\end{equation}
So for each even $T \in \CJ[2]$, there is a one-dimensional eigenspace
of quadratic forms such that the action of~$T'$ is given by multiplication
with~$e_2(T, T')$. For $T = 0$, this eigenspace is spanned by the invariant
quadratic~\eqref{quadrel}.

\begin{Definition}
  We set
  \[ y_0 = 2(\xi_1 \xi_8 - \xi_2 \xi_7 + \xi_3 \xi_6 - \xi_4 \xi_5) ; \]
  this is the quadratic form corresponding to~$S$, since
  $y_0(\vxi) = \vxi S \vxi^\top = \langle \vxi, \vxi \rangle_S$.
  For nontrivial even~$T$, we denote by~$y_T$ the
  form in the eigenspace corresponding to~$T$ that has coefficient~$1$ on~$\xi_8^2$.
  We will see that this makes sense, i.e., that this coefficient is always nonzero.
\end{Definition}

\begin{Lemma} \label{L:y}
  For every nontrivial even $2$-torsion point~$T$, the matrix corresponding
  to the quadratic form~$y_T$ is the symmetric matrix $S M_T$.
  In particular, if $T$ corresponds to a factorization
  $F = GH$ into two polynomials of degree~4, then the coefficients of~$y_T$ are
  polynomials in the coefficients of~$G$ and~$H$ with integral coefficients, and
  the coefficients of the monomials $\xi_i \xi_j$ with $i \neq j$ are
  divisible by~2.
\end{Lemma}

\begin{proof}
  We show that
  $\tilde{M}_{T'}^\top (S M_T) \tilde{M}_{T'} = e_2(T,T') S M_T$. We use that
  $\tilde{M}_{T'}^2 = \eps(T') I_8$,
  $S \tilde{M}_{T'} = \eps(T') \tilde{M}_{T'}^\top S$
  and the fact that the Weil pairing is given by commutators. This gives that
  \[ \tilde{M}_{T'}^\top S M_T \tilde{M}_{T'}
      = \eps(T') S \tilde{M}_{T'} M_T \tilde{M}_{T'}
      = \eps(T') e_2(T,T') S M_T \tilde{M}_{T'}^2
      = e_2(T,T') S M_T
  \]
  as desired, so $S M_T$ gives a quadratic form in the correct eigenspace.
  Since the upper right entry of~$M_T$ is~$1$, the lower right entry, which
  corresponds to the coefficient of~$\xi_8^2$, of~$S M_T$ is~$1$, so that
  we indeed obtain~$y_T$.
\end{proof}

We can express $y_T$ as $y_T(\vxi) = \langle \vxi, \vxi M_T^\top \rangle_S$.

\begin{Remark}
  Note that if $T$ is an odd $2$-torsion point, represented by the factorization
  $(G,H)$, then the same argument shows that the alternating bilinear form
  corresponding to the matrix $S M_{(G,H)}$ is multiplied by~$e_2(T,T')$ under
  the action of~$T' \in \CJ[2]$.
\end{Remark}

We set
\[ (\eps_1, \eps_2, \eps_3, \eps_4, \eps_5, \eps_6, \eps_7, \eps_8)
      = (1, -1, 1, -1, -1, 1, -1, 1) ;
\]
these are the entries occurring in~$S$ along the diagonal from upper right
to lower left.

\begin{Corollary} \label{Cor:y}
  Let $T$ be a nontrivial even $2$-torsion point with image on~$\CK$
  given by
  \[ (1 : \tau_2 : \tau_3 : \tau_4 : \tau_5 : \tau_6 : \tau_7 : \tau_8) . \]
  Then
  \[ y_T = \xi_8^2 + 2 \sum_{j=2}^8 \eps_j \tau_j  \xi_{9-j} \xi_8
             + (\text{terms not involving~$\xi_8$}) .
  \]
\end{Corollary}

A similar statement is true for $T = 0$ if we take coordinates $(0 : \ldots : 0 : 1)$:
we have that $y_0 = 2 \xi_1 \xi_8 + (\text{terms not involving~$\xi_8$})$.

\begin{proof}
  The last column of~$M_T$ has entries $1, \tau_2, \ldots, \tau_8$
  (since $M_T$ maps the origin to the image of~$T$ and has upper right entry~$1$).
  Multiplication by~$S$
  from the left reverses the order and introduces the signs~$\eps_j$.
  Since the coefficients of~$y_T$ of monomials involving~$\xi_8$ are
  given by the entries of the last column of~$S M_T$ by Lemma~\ref{L:y},
  the claim follows.
\end{proof}

We define a pairing on the space $V_1 \otimes V_1$ of bilinear forms
as follows. If the bilinear forms $\phi$ and~$\phi'$ are represented by
matrices $A$ and~$A'$ with respect to our standard basis $\xi_1, \ldots, \xi_8$
of~$V_1$, then $\langle \phi, \phi' \rangle = \tfrac{1}{8} \Tr(A^\top A')$
(the scaling has the effect of giving the standard quadratic form norm~1).

For an even $2$-torsion point~$T$,
we write $\tilde{y}_T$ for the symmetric bilinear form corresponding to
the matrix $S \tilde{M}_T$ (this is well-defined up to sign) and
$\tilde{z}_T$ for the symmetric bilinear form corresponding
to $S \tilde{M}_T^\top = \tilde{M}_T S$.
Also, $z_T$ will denote the form corresponding to $S M_T^\top = M_T S$.
Then, since $S (M_T S) S = S M_T$, we have the relation
$z_T(\vxi) = y_T(\vxi S)$; explicitly,
\[ z_T(\xi_1, \xi_2, \xi_3, \xi_4, \xi_5, \xi_6, \xi_7, \xi_8)
     = y_T(\xi_8, -\xi_7, \xi_6, -\xi_5, -\xi_4, \xi_3, -\xi_2, \xi_1) .
\]

\begin{Lemma} \label{L:ortho}
  For all even $2$-torsion points $T$ and~$T'$, we have that
  \[ \langle \tilde{z}_T, \tilde{y}_{T'} \rangle
       = \begin{cases}
           1 & \text{if $T = T'$,} \\
           0 & \text{if $T \neq T'$.}
         \end{cases}
  \]
  Equivalently,
  \[ \langle z_T, y_{T'} \rangle
       = \begin{cases}
           r(T) & \text{if $T = T'$,} \\
           0 & \text{if $T \neq T'$.}
         \end{cases}
  \]
  Here we restrict the scalar product defined above to $V_2 \subset V_1 \otimes V_1$.
\end{Lemma}

\begin{proof}
  The claim is that $\Tr\bigl((S \tilde{M}_T^\top)^\top (S \tilde{M}_{T'})\bigr)$
  is zero if $T \neq T'$ and equals~$8$ if $T = T'$. We have
  \[ \Tr\bigl((S \tilde{M}_T^\top)^\top (S \tilde{M}_{T'})\bigr)
       = \Tr(\tilde{M}_T S^2 \tilde{M}_{T'})
       = \Tr(\tilde{M}_T \tilde{M}_{T'})
       = \pm \Tr(\tilde{M}_{T+T'}) .
  \]
  If $T \neq T'$, then this trace is zero, as we had already seen.
  If $T = T'$, then $\pm \tilde{M}_{T+T'} = I_8$, so the result is~$8$ as desired.
\end{proof}

This allows us to express the $\xi_j^2$ in terms of the~$y_T$.
We set $r(0) = 1$ and $M_0 = I_8$.
We denote the coefficient of~$\xi_i \xi_j$ in a quadratic form $q \in V_2$
by $[\xi_i \xi_j] q$.

\begin{Lemma} \label{L:x2iny}
  For every $j \in \{1,2,\ldots,8\}$, we have that
  \[ \xi_j^2 = \sum_{T \colon \eps(T) = 1} \frac{[\xi_{9-j}^2] y_T}{8 r(T)} y_T . \]
  Similarly, for $1 \le i < j \le 8$, we have that
  \[ 2 \xi_i \xi_j = \eps_i \eps_j \sum_{T \colon \eps(T) = 1}
                       \frac{[\xi_{9-i} \xi_{9-j}] y_T}{8 r(T)} y_T .
  \]
\end{Lemma}

\begin{proof}
  We have by Lemma~\ref{L:ortho} that
  \begin{align*}
    \xi_j^2 &= \sum_{T \colon \eps(T) = 1} \langle \tilde{z}_T, \xi_j^2 \rangle \tilde{y}_T
             = \sum_{T \colon \eps(T) = 1} \frac{\langle z_T, \xi_j^2 \rangle}{r(T)} y_T \\
            &= \sum_{T \colon \eps(T) = 1} \frac{[\xi_j^2] z_T}{8 r(T)} y_T
             = \sum_{T \colon \eps(T) = 1} \frac{[\xi_{9-j}^2] y_T}{8 r(T)} y_T .
  \end{align*}
  In the same way, we have for $i \neq j$ that
  \begin{align*}
    2 \xi_i \xi_j
      &= \sum_{T \colon \eps(T) = 1} 2 \langle \tilde{z}_T, \xi_i \xi_j \rangle \tilde{y}_T
       = \sum_{T \colon \eps(T) = 1} 2 \frac{\langle z_T, \xi_i \xi_j \rangle}{r(T)} y_T \\
      &= \sum_{T \colon \eps(T) = 1} \frac{[\xi_i \xi_j] z_T}{8 r(T)} y_T
       =  \eps_i \eps_j\sum_{T \colon \eps(T) = 1}
                                \frac{[\xi_{9-i} \xi_{9-j}] y_T} {8 r(T)} y_T .
  \end{align*}
  (Note that $8 \langle z_T, \xi_i \xi_j \rangle$ is half the coefficient
  of $\xi_i \xi_j$ in~$z_T$.)
\end{proof}

\begin{Corollary} \label{C:sumyT2}
  We have that
  \begin{align*}
    \sum_{T \colon \eps(T) = 1} \frac{1}{8 r(T)} y_T(\vxi) y_T(\vz)
       = \Bigl(\sum_{j=1}^8 \eps_j  \xi_j \zeta_{9-j}\Bigr)^2
       = \langle \vxi, \vz \rangle_S^2 .
  \end{align*}
  In particular, setting $\vz = \vxi$, we obtain that
  \[ \sum_{T \colon \eps(T) = 1} \frac{1}{8 r(T)} y_T^2
      = y_0^2
      = 4 (\xi_1 \xi_8 - \xi_2 \xi_7 + \xi_3 \xi_6 - \xi_4 \xi_5)^2 .
  \]
\end{Corollary}

\begin{proof}
  We compute using Lemma~\ref{L:x2iny}:
  \begin{align*}
    \sum_{T \colon \eps(T) = 1} &\frac{1}{8 r(T)} y_T(\vxi) y_T(\vz) \\
      &= \sum_{i=1}^8 \xi_i^2
           \sum_{T \colon \eps(T) = 1} \frac{[\xi_i^2] y_T(\vxi)}{8 r(T)} y_T(\vz)
          + \sum_{1 \le i < j \le 8} \xi_i \xi_j
              \sum_{T \colon \eps(T) = 1} \frac{[\xi_i \xi_j] y_T(\vxi)}{8 r(T)} y_T(\vz) \\
      &= \sum_{i=1}^8 \xi_i^2 \zeta_{9-i}^2
          + 2 \sum_{1 \le i < j \le 8} \eps_i \eps_j \, \xi_i \xi_j \zeta_{9-i} \zeta_{9-j} \\
      &= \Bigl(\sum_{j=1}^8 \eps_j \, \xi_j \zeta_{9-j}\Bigr)^2 . \qedhere
  \end{align*}
\end{proof}

Now we consider the representation~$\rho_4$ of~$\CJ[2]$ on the space~$V_4$
of quartic forms. For its character~$\chi_4$, we have the general formula
\[ \chi_4(T) = \tfrac{1}{24}\bigl(
                 \chi_1(\tilde{M}_T)^4 + 8 \chi_1(\tilde{M}_T) \chi_1(\tilde{M}_T^3)
                   + 3 \chi_1(\tilde{M}_T^2)^2
                   + 6 \chi_1(\tilde{M}_T)^2 \chi_1(\tilde{M}_T^2)
                   + 6 \chi_1(\tilde{M}_T^4)\bigr) .
\]
This gives us that
\[ \chi_4(0) = 330 \qquad\text{and}\qquad
   \chi_4(T) = 10 \quad\text{for $T \neq 0$.}
\]
We deduce that
\begin{equation} \label{eqn:rho4}
  \rho_4 = \chi_0^{\oplus 15} \oplus \bigoplus_{T \neq 0} \chi_T^{\oplus 5} .
\end{equation}


\section{The duplication map and the missing generator of $L(4\Theta)^+$} \label{S:dup}

We continue to work over a field~$k$ of characteristic~$\neq 2$.
We also continue to assume that $F \in k[x,z]$ is squarefree, so that $\CC$ is a smooth
hyperelliptic curve of genus~$3$ over~$k$.

Consider the commutative diagram
\[ \xymatrix{\CJ \ar[r]^{\cdot 2} \ar[d]^{\kappa} & \CJ \ar[d]^{\kappa} \\
             \CK \ar[r]^{\delta} & \CK \ar@{^(->}[r] & \BP^7,
            }
\]
where the map in the top row is multiplication by~$2$ and $\delta$ is the
endomorphism of~$\CK$ induced by it. Pulling back a hyperplane section
to the copy of~$\CJ$ on the right, we obtain a divisor in the class of~$2\Theta$.
Pulling it further back to the copy on the left, we obtain a divisor in the class
of the pull-back of~$2\Theta$ under duplication, which is the class of~$8\Theta$
($\Theta$ is symmetric, so pulling back under multiplication by~$n$ multiplies
its class by~$n^2$). The combined map from the left~$\CJ$ to~$\BP^7$ then
is given by an $8$-dimensional subspace of~$L(8\Theta)^+$; by Corollary~\ref{C:sym4}
this means that~$\delta$ is given by eight quartic forms in~$\vxi$.
Since $\delta$ maps $o$, the image of the origin on~$\CK$, to itself, we
can normalize these quartics so that they evaluate to $(0,\ldots,0,1)$ on~$(0,\ldots,0,1)$.
We use $\vd = (\delta_1, \ldots, \delta_8)$ to denote these quartic forms;
they are determined up to adding a quartic form vanishing on~$\CK$.
We write $E_4 \subset V_4$ for the subspace of quartics vanishing
on~$\CK$. Note that we can test whether a given homogeneous polynomial in~$\vxi$
vanishes on~$\CK$ by pulling it back to~$\CW$ or to~$\BA^{15}$ and checking
whether it vanishes on~$\CV$.

We now determine the structure of~$E_4$ as a representation of~$\CJ[2]$
and we identify the space generated by~$\vd$ in $V_4/E_4$.

\begin{Lemma} \label{L:quartics} \strut
  \begin{enumerate}[\upshape(1)]
    \item The restriction of~$\rho_4$ to~$E_4$ splits as
          $\rho_4|_{E_4} = \chi_0^{\oplus 7} \oplus \bigoplus_{T \neq 0} \chi_T$.
    \item The images of $\delta_1, \ldots, \delta_8$ form a basis of the quotient
          $V_4^{\CJ[2]}/E_4^{\CJ[2]}$ of invariant subspaces.
  \end{enumerate}
\end{Lemma}

\begin{proof} \strut
  \begin{enumerate}[(1)]
    \item The dimension of~$E_4$ is~$70$ by Theorem~\ref{T:Kummerrel},
          and a subspace of dimension~$36$
          is given by $y_0 V_2$. The latter splits in the same way as~$\rho_2$ does.
          Since for the generic curve, the Galois action is transitive
          on the odd $2$-torsion points and on the nontrivial even
          $2$-torsion points, the multiplicities of all odd characters
          and those of all nontrivial even characters in~$\rho_4|_{E_4}$ have to
          agree. The only way to make the numbers come out correctly is as indicated.
    \item Since the result of duplicating a point is unchanged when a $2$-torsion
          point is added to it, the images of all~$\delta_j$ in~$V_4/E_4$
          must lie in the same eigenspace of the $\CJ[2]$-action. Since $K$ spans~$\BP^7$
          and the duplication map $\delta \colon K \to K$ is surjective, the images
          of the~$\delta_j$ in~$V_4/E_4$ must be linearly independent.
          So they must live in an eigenspace of dimension at least~eight.
          The only such eigenspace is that of the trivial character, which
          has dimension exactly $8 = 15 - 7$ by the first part. \qedhere
  \end{enumerate}
\end{proof}

We see that the 36 quartic forms $y_T^2$ for $T$ an even $2$-torsion point
are in the invariant subspace of~$V_4$ of dimension~15.
Let $\CT_\even$ denote the finite $k$-scheme
whose geometric points are the 36~even $2$-torsion points
(we can consider $\CT_\even$ as a subscheme of~$\CJ$ or of~$\CK$), and denote
by~$k_\even$ its coordinate ring; this is an \'etale $k$-algebra of dimension~36.
Then $y \colon T \mapsto y_T$ can be considered as a quadratic form with coefficients
in~$k_\even$ and $r \colon T \mapsto r(T)$ is an element of $k_\even^\times$.

\begin{Lemma} \label{L:etalg}
  The 36~coefficients $c_{ii} = [\xi_i^2]y$, for $1 \le i \le 8$,
  and $c_{ij} = \half [\xi_i \xi_j]y$,
  for $1 \le i < j \le 8$, constitute a $k$-basis of~$k_\even$.
\end{Lemma}

\begin{proof}
  We define further elements of $k_\even$ by
  \[ \tilde{c}_{ii} = \frac{1}{8 r} [\xi_{9-i}^2]y \qquad\text{and}\qquad
     \tilde{c}_{ij} = \frac{\eps_i \eps_j}{8 r} [\xi_{9-i} \xi_{9-j}]y .
  \]
  Lemma~\ref{L:x2iny} can be interpreted as saying that
  \[ \Tr_{k_\even/k} (\tilde{c}_{ij} c_{i'j'})
       = \begin{cases}
           1 & \text{if $(i,j) = (i',j')$,} \\
           0 & \text{otherwise.}
         \end{cases}
  \]
  This shows that the given elements are linearly independent over~$k$.
\end{proof}

We can compute the structure constants of~$k_\even$ with respect to this
basis and use this to express $y^2$ in terms of the basis again.
Extracting coefficients, we obtain 36~quartic forms with coefficients
in~$k$ that all lie in the 15-dimensional space of invariants under~$\CJ[2]$.
We check\star\ that they indeed span a space of this dimension and that we get
a subspace of dimension~7 of quartics vanishing on the Kummer variety.

\begin{figure}[t]
\hrulefill
\begin{align*}
  q_1  &= \xi_1 \xi_8^3
           + 2(-f_2 \xi_2 + f_3 \xi_3 - f_4 \xi_4 - f_4 \xi_5 + f_5 \xi_6 - f_6 \xi_7)
                \xi_1 \xi_8^2
           + \ldots \\
  q_2  &= \xi_2 \xi_8^3
           + (4 f_8 (-f_0 \xi_2 + f_2 \xi_4 + f_4 \xi_7)
               - 2 f_3 f_8 \xi_6 - f_5 f_7 \xi_7) \xi_1 \xi_8^2
           + \ldots \\
  q_3  &= \xi_3 \xi_8^3
           + (f_7 (-2 f_0 \xi_2 + 2 f_2 \xi_4 + f_3 \xi_6) \\
       & \hphantom{{}= \xi_3 \xi_8^3 + ({}}
              + 2 f_8 (-2 f_0 \xi_3 + 4 f_1 \xi_4 - 2 f_2 \xi_6 - f_3 \xi_7))
                \xi_1 \xi_8^2
           + \ldots \\
  q_4  &= \xi_4 \xi_8^3
           + (-2 f_0 f_7 \xi_3 + (12 f_0 f_8 + f_1 f_7) \xi_4 - 2 f_1 f_8 \xi_6)
                \xi_1 \xi_8^2
           + \ldots \\
  q_5  &= \xi_5 \xi_8^3
           + ((4 f_0 f_6 - 2 f_1 f_5) \xi_2
               + (-2 f_0 f_7 - 2 f_1 f_6 + 2 f_2 f_5) \xi_3 \\
       & \hphantom{{}= \xi_5 \xi_8^3 + ({}}
               + (4 f_0 f_8 + 4 f_1 f_7 + 4 f_2 f_6 - 5 f_3 f_5) \xi_4 \\
       & \hphantom{{}= \xi_5 \xi_8^3 + ({}}
               + (-2 f_1 f_8 - 2 f_2 f_7 + 2 f_3 f_6) \xi_6
               + (4 f_2 f_8 - 2 f_3 f_7) \xi_7) \xi_1 \xi_8^2
           + \ldots \\
  q_6  &= \xi_6 \xi_8^3
           + (f_0 (-2 f_5 \xi_2 - 4 f_6 \xi_3 + 8 f_7 \xi_4 - 4 f_8 \xi_6) \\
       & \hphantom{{}= \xi_6 \xi_8^3 + ({}}
               + f_1 (f_5 \xi_3 + 2 f_6 \xi_4 - 2 f_8 \xi_7)) \xi_1 \xi_8^2
           + \ldots \\
  q_7  &= \xi_7 \xi_8^3
           + (4 f_0 (f_4 \xi_2 + f_6 \xi_4 - f_8 \xi_7)
               - f_1 f_3 \xi_2 - 2 f_0 f_5 \xi_3) \xi_1 \xi_8^2
           + \ldots \\
  q_8  &= \xi_8^4
           + 16 (f_1 f_8 (f_1 \xi_2 - f_2 \xi_3 + f_3 \xi_4)
                  + f_0 f_7 (f_5 \xi_4 - f_6 \xi_6 + f_7 \xi_7)) \xi_1 \xi_8^2
           + \ldots \\[3mm]
  q_9  &= 2 (f_7 \xi_6 - 4 f_8 \xi_7) \xi_1 \xi_8^2 + \ldots \\
  q_{10} &= 2 (f_5 \xi_4 - f_6 \xi_6 + f_7 \xi_7) \xi_1 \xi_8^2 + \ldots \\
  q_{11} &= 2 (f_3 \xi_3 + 2 f_4 \xi_4 - 2 f_4 \xi_5 + f_5 \xi_6) \xi_1 \xi_8^2
              + \ldots \\
  q_{12} &= 2 (f_1 \xi_2 - f_2 \xi_3 + f_3 \xi_4) \xi_1 \xi_8^2 + \ldots \\
  q_{13} &= 2 (-4 f_0 \xi_2 + f_1 \xi_3) \xi_1 \xi_8^2 + \ldots \\
  q_{14} &= (3 \xi_4 - \xi_5) \xi_1 \xi_8^2 + \ldots \\
  q_{15} &= \xi_1^2 \xi_8^2 + \ldots
          = (\xi_1 \xi_8 - \xi_2 \xi_7 + \xi_3 \xi_6 - \xi_4 \xi_5)^2
\end{align*}
\caption{A basis of the $\CJ[2]$-invariant subspace of~$V_4$.} \label{Fig:qs}

\hrulefill
\end{figure}

It turns out\star\ that the quartics in $V_4^{\CJ[2]}$ that vanish on~$\CK$
are exactly those that do not contain terms cubic or quartic in~$\xi_8$.
Forms spanning the complementary space are uniquely determined
modulo $E_4^{\CJ[2]}$ by fixing the terms of higher degree in~$\xi_8$.
We take $q_j = \xi_j \xi_8^3 + (\deg_{\xi_8} \le 2)$ for $j = 1, \ldots, 8$.
Then the $q_j$ can be chosen so that they have coefficients in~$\Z[f_0,\ldots,f_8]$.
To fix~$q_j$ completely, it suffices to specify in addition the coefficients
of $\xi_1 \xi_i \xi_8^2$ for $1 \le i \le 7$. One possibility is to choose
them as given in Figure~\ref{Fig:qs}, which
includes $q_9, \ldots, q_{15}$ in the ideal of~$\CK$, where
$E_4^{\CJ[2]} = \langle q_9, q_{10}, \ldots, q_{15} \rangle$.
These quartics can be obtained from \texttt{Kum3-invariants.magma} at~\cite{Files}.

We can now identify the duplication map on~$\CK$.

\begin{Theorem} \label{Thm:delta}
  The polynomials
  \[ (\delta_1, \delta_2, \delta_3, \delta_4, \delta_5, \delta_6, \delta_7, \delta_8)
      = (4 q_1, 4 q_2, 4 q_3, 4 q_4, 4 q_5, 4 q_6, 4 q_7, q_8)
  \]
  in~$V_4^{\CJ[2]}$ (with $q_j$ as above) have the following properties.
  \begin{enumerate}[\upshape(1)]\addtolength{\itemsep}{1mm}
    \item $\delta_j \in \Z[f_0,f_1,\ldots,f_8][\xi_1,\xi_2,\ldots,\xi_8]$
          for all $1 \le j \le 8$.
    \item $(\delta_1, \delta_2, \ldots, \delta_8)(0,0,\ldots,0,1)
             = (0,0,\ldots,0,1)$.
    \item With $y_T$ as defined earlier for an even $2$-torsion point
          with image
          \[ (1 : \tau_2 : \tau_3 : \tau_4 : \tau_5 : \tau_6 : \tau_7 : \tau_8) \]
          on~$\CK$, we have that
          \[ y_T^2 \equiv \delta_8 - \tau_2 \delta_7 + \tau_3 \delta_6
                              - \tau_4 \delta_5 - \tau_5 \delta_4 + \tau_6 \delta_3
                              - \tau_7 \delta_2 + \tau_8 \delta_1
                   = \langle \vtau, \vd \rangle_S
                          \bmod E_4^{\CJ[2]} ,
          \]
          where $\vtau = (1, \tau_2, \ldots, \tau_8)$ and $\vd = (\delta_1, \ldots, \delta_8)$.
    \item The $\delta_j$ do not vanish simultaneously on~$\CK$.
    \item The map $\delta \colon \CK \to \CK$ given by $(\delta_1 : \ldots : \delta_8)$
          is the duplication map on~$\CK$.
  \end{enumerate}
\end{Theorem}

\begin{proof} \strut
  \begin{enumerate}[(1)]
    \item This can be verified using the explicit polynomials.
    \item This is obvious.
    \item We compare the coefficients of $\xi_j \xi_8^3$ on both sides. Since
          by Corollary~\ref{Cor:y},
          \[ y_T = \xi_8^2 + 2 \eps_2 \tau_2 \xi_7 \xi_8 + 2 \eps_3 \tau_3 \xi_6 \xi_8
                     + \ldots + 2 \eps_8 \tau_8 \xi_1 \xi_8
                     + (\text{terms not involving~$\xi_8$}) ,
          \]
          we find that
          \[ y_T^2 = \xi_8^4 + 4 \eps_2 \tau_2 \xi_7 \xi_8^3 + \ldots
                             + 4 \eps_8 \tau_8 \xi_1 \xi_8^3
                             + (\text{terms of degree $\le 2$ in~$\xi_8$})
          \]
          and the right hand side has the same form. So the difference
          is a form in~$V_4^{\CJ[2]}$ of degree at most~2 in~$\xi_8$, which
          implies that it is in~$E_4^{\CJ[2]}$.
    \item Let $\vxi \in k^8 \setminus \{0\}$ be coordinates of a point in~$\CK$.
          Then $\vd(\vxi) = 0$
          implies by~(3) that $y_T(\xi) = 0$ for all even $2$-torsion points~$T$
          (note that $y_0$ vanishes on all of~$\CK$). Lemma~\ref{L:x2iny} then
          shows that $\vxi = 0$ as well, since $8 r(T) \neq 0$ in~$k$.
          This contradicts our choice of~$\vxi$.
    \item By~(4), $\delta$ is a morphism $\CK \to \BP^7$, and by
          Lemma~\ref{L:quartics}~(2) $\delta$ differs from the duplication
          map by post-composing with an automorphism~$\alpha$ of~$\BP^7$.
          We show\star\ that on a generic point, $\delta$ coincides with the duplication
          map; this proves that $\alpha$ is the identity.
          We use the action of~$\GL(2)$ on $(x,z)$ (and scaling on~$y$) to reduce to the case
          that $F(x,1)$ is monic of degree~$7$. A generic point~$P$ on~$\CJ$
          can then be represented by $(A,B,C)$ such that $A(x,1)$ is monic
          of degree~$3$ and squarefree and $B(x,1)$ is of degree $\le 2$.
          After making a further affine transformation, we can assume that
          $A(x,1) = x(x-1)(x-a)$ for some $a \in k$. The corresponding point on~$\CK$
          is then
          \begin{align*}
             \kappa(P) = (1 &: -a-1 : a : 0 : -a c_3 - c_1 : -c_0 \\
                            &: (a+1) c_0 + 2 b_0 b_2 : -(a^2+a+1) c_0 - 2(a+1) b_0 b_2) ,
          \end{align*}
          where $B(x,1) = b_0 + b_1 x + b_2 x^2$,
          $C(x,1) = c_0 + c_1 x + c_2 x^2 + c_3 x^3 - x^4$.
          We compute~$2P$ in terms of its Mumford representation using
          Cantor's algorithm as implemented in Magma and find~$\kappa(2P)$.
          On the other hand, we compute~$\delta(\kappa(P))$. Both points
          are equal, which proves the claim.
    \qedhere
  \end{enumerate}
\end{proof}
The quartics $\vd = (\delta_1, \ldots, \delta_8)$
are given in the file \texttt{Kum3-deltas.magma} at~\cite{Files}.

The canonical map from $V_2 = \Sym^2 L(2 \Theta)$ to $L(4 \Theta)$ has non-trivial
one-di\-men\-sio\-nal kernel, spanned by the quadric~$y_0$ vanishing on~$\CK$. Since
the dimension of the even part~$L(4 \Theta)^+$ of~$L(4 \Theta)$ is $36 = \dim V_2$,
the map $V_2 \to L(4 \Theta)^+$ has a one-dimensional cokernel. Looking at the
action of~$\CJ[2]$ on~$L(4 \Theta)^+$, it is clear that this space splits as
a direct sum of the image of~$V_2$ and a one-dimensional invariant subspace.
We will identify a generator of the latter.

\begin{Lemma}
  The image of $q_1$ in $L(8 \Theta)$ is the square of an element
  $\Xi \in L(4 \Theta)^+$ that is invariant under the action of~$\CJ[2]$.
\end{Lemma}

\begin{proof}
  We pull back $q_1$ to a polynomial
  function on the affine space~$\BA^{15}$ that parameterizes the triples
  of polynomials $(A,B,C)$. We find\star\ that this polynomial is the square
  of some other polynomial~$p$ that can be written as a quadratic in the
  components of $A_l \wedge B_l \wedge C_l$.
  So $p$ is invariant under~$\pm\Gamma$, which means that it gives an
  element~$\Xi$ of~$L(4 \Theta)^+$.
\end{proof}

To make~$\Xi$ more explicit, we note that $p$ can be expressed
as a cubic in the~$\xi_j$. Taking into account that $\xi_1 = 1$
on the affine space, we find that (up to the choice of a sign)
\begin{align*}
  \xi_1 \Xi
    &= (-8 f_0 f_4 f_8 + 2 f_0 f_5 f_7 + 2 f_1 f_3 f_8) \xi_1^3
        - 4 f_0 f_6 \xi_1^2 \xi_2
        + (-4 f_0 f_7 + 2 f_1 f_6) \xi_1^2 \xi_3 \\
    &\qquad{}
        + (-4 f_0 f_8 + 2 f_1 f_7 - 4 f_2 f_6 + f_3 f_5) \xi_1^2 \xi_4
        + (12 f_0 f_8 - f_1 f_7) \xi_1^2 \xi_5 \\
    &\qquad{}
        + (-4 f_1 f_8 + 2 f_2 f_7) \xi_1^2 \xi_6
        - 4 f_2 f_8 \xi_1^2 \xi_7
        + 6 f_0 \xi_1 \xi_2^2 - 3 f_1 \xi_1 \xi_2 \xi_3
        + 6 f_2 \xi_1 \xi_2 \xi_4 \\
    &\qquad{}
        - f_3 \xi_1 \xi_2 \xi_6
        - 2 f_3 \xi_1 \xi_3 \xi_4
        + 2 f_4 \xi_1 \xi_3 \xi_6
        - f_5 \xi_1 \xi_3 \xi_7
        + 4 f_4 \xi_1 \xi_4^2
        - 2 f_4 \xi_1 \xi_4 \xi_5 \\
    &\qquad{}
        - 2 f_5 \xi_1 \xi_4 \xi_6
        + 6 f_6 \xi_1 \xi_4 \xi_7
        - 3 f_7 \xi_1 \xi_6 \xi_7
        + 6 f_8 \xi_1 \xi_7^2
        - 11 \xi_2 \xi_4 \xi_7
        + \xi_2 \xi_5 \xi_7
        + 2 \xi_2 \xi_6^2 \\
    &\qquad{}
        + 2 \xi_3^2 \xi_7
        + 5 \xi_3 \xi_4 \xi_6
        - 3 \xi_3 \xi_5 \xi_6
        + 2 \xi_4^3
        - 7 \xi_4^2 \xi_5
        + 3 \xi_4 \xi_5^2 .
\end{align*}
We obtain similar cubic expressions for $\xi_j \Xi$ with
$j \in \{2,3,\ldots,8\}$ by multiplying the polynomial above
by~$\xi_j$, then adding a suitable linear combination of the
quartics vanishing on~$\CK$ so that we obtain something that
is divisible by~$\xi_1$.
These cubics are given in the file \texttt{Kum3-Xipols.magma} at~\cite{Files}.
With this information, we can evaluate~$\Xi$ on any given set~$\vxi$
of coordinates of a point on~$\CK$: we find an index~$j$ with $\xi_j \neq 0$
and evaluate $\Xi$ as $(\xi_j \Xi)/\xi_j$.

This gives us a basis of $L(4 \Theta)^+$ consisting of~$\Xi$ and the quadratic
monomials in the~$\xi_j$ minus one of the monomials $\xi_j \xi_{9-j}$.
Alternatively, we can use the basis consisting of~$\Xi$ and the~$y_T$ for
the 35 nonzero even $2$-torsion points~$T$.


\section{Sum and difference on the Kummer variety} \label{S:sumdiff}

In this section, $k$ continues to be a field of characteristic~$\neq 2$ and
$F$ to be squarefree.

We consider the composition
\[ \CJ \times \CJ \stackrel{(+,-)}{\To} \CJ \times \CJ
     \stackrel{(\kappa,\kappa)}{\To} \CK \times \CK \To \BP^7 \times \BP^7
     \stackrel{\text{Segre}}{\To} \BP^{63}
     \stackrel{\text{symm.}}{\To} \BP^{35}
\]
where `symm.' is the symmetrization map that sends a matrix~$A$
to $A + A^\top$ and we identify the Segre map with the multiplication map
\[ \text{(column vectors)} \times \text{(row vectors)} \To \text{matrices}. \]
Pulling back hyperplanes to $\CJ \times \CJ$, we see that the map is given by sections
of $4 \pr_1^* \Theta + 4 \pr_2^* \Theta$, hence symmetric bilinear forms
on~$L(4 \Theta)$. The map is invariant under
negation of either one of the arguments, therefore the bilinear forms
only involve even sections. The map can be described by a symmetric
matrix~$B$ of such bilinear forms such that in terms of coordinates $(w_j)$
and~$(z_j)$ of the images $\kappa(P+Q)$ and~$\kappa(P-Q)$ of $P \pm Q$
on~$\CK$, we have (up to scaling) $w_i z_j + w_j z_i = 2 B_{ij}(\kappa(P), \kappa(Q))$.
We normalize by requiring that $B_{88}(o,o) = 1$, where $o = (0,\ldots,0,1)$.

We write $\tilde{V}_2$ for~$L(4 \Theta)^+$; then $B$ can be interpreted
as an element~$\beta$ of $\tilde{V}_2 \otimes \tilde{V}_2 \otimes V_2^*$.
The last factor $V_2^*$ is identified with the space of symmetric $8 \times 8$
matrices (whose entries are thought of representing $\half(w_i z_j + w_j z_i)$
for coordinates $\underline{w}$ and~$\underline{z}$ of points in~$\BP^7$)
by specifying that a quadratic form $q \in V_2$ evaluates on such a matrix
to $b(\underline{w}, \underline{z})$ where $b$ is the bilinear form such that
$q(\vx) = b(\vx, \vx)$. If $M$ is the matrix of~$b$ and $B$ is the matrix
corresponding to the unordered pair $\{\underline{w}, \underline{z}\}$, then
the pairing is $\Tr(M^\top B) = 8 \langle M, B \rangle$. Put differently,
we obtain the $(i,j)$-entry of the matrix by evaluating at the quadratic
form~$\xi_i \xi_j$.

The $2$-torsion group~$\CJ[2]$ acts on each factor, and $\beta$ must be invariant
under the action of $\CJ[2] \times \CJ[2]$ such that $(T, T')$ acts via
$(T, T', T+T')$ on the three factors (shifting $P$ by~$T$ and $Q$ by~$T'$
shifts $P \pm Q$ by~$T+T'$).

We use the basis of~$\tilde{V}_2$ given by $\Xi$
and $y_T$ for the nonzero even $2$-torsion points~$T$ (suitably extending~$k$
if necessary); for $V_2^*$ we
use the basis dual to $(y_T)_{\text{$T$ even}}$, which is given by the linear forms
\[ y_T^* \colon v \longmapsto \frac{1}{r(T)} \langle z_T, v \rangle . \]
If $T_1, T_2, T_3$ are even $2$-torsion points, then the effect of $(T,T')$
acting on the corresponding basis element of the triple tensor product is
to multiply it by
\[ e_2(T, T_1) e_2(T', T_2) e_2(T+T', T_3) = e_2(T, T_1 + T_3) e_2(T', T_2 + T_3) . \]
If this basis element occurs in~$\beta$ with a nonzero coefficient, then
this factor must be~1 for all $T, T'$, which means that $T_1 = T_2 = T_3$.
This shows that
\[ \beta = \sum_{T \neq 0} a_T (y_T \otimes y_T \otimes y_T^*)
              + a_0 (\Xi \otimes \Xi \otimes y_0^*) .
\]
If we evaluate at the origin in the first component, we obtain (using that
$\Xi$ vanishes there and that $y_T(o) = 1$ for $T \neq 0$ even) that
\[ \beta_o = \sum_{T \neq 0} a_T (y_T \otimes y_T^*) . \]
This corresponds to taking $P = O$, resulting in the pair $\pm Q$ leading
to $\{\kappa(Q), \kappa(Q)\}$. So, taking $\vxi$ as coordinates of~$Q$
and using that $B_{88}(o, o) = 1$,
the $(i,j)$-component of this expression, evaluated at $\vxi$ in the
(now) first component of~$\beta_o$,
must be~$\xi_i \xi_j$, up to a multiple of~$y_0$:
\[ \xi_i \xi_j \equiv \sum_{T \neq 0} a_T y_T^*(\xi_i \xi_j) \cdot y_T \bmod y_0 . \]
In other words, $\beta_o$, interpreted as a linear map $V_2 \to \tilde{V}_2$,
is the canonical map; in particular, it sends $y_T$ to~$y_T$ for all even $T \neq 0$,
and so $a_T = 1$ for all $T \neq 0$. It only remains to find~$a_0$; then $\beta$
is completely determined. We consider the image of~$\beta$
in $\Sym^2 \tilde{V}_2 \otimes V_2^*$,
which corresponds to taking $P = Q$. This results in the unordered pair $\{2P, O\}$,
represented (according to our normalization) by the symmetric matrix that is zero everywhere
except in the last row and column, where it has entries
$\half\delta_1, \ldots, \half\delta_7, \delta_8$. We obtain (recall that $\Xi^2 = q_1$
and $\delta_1 = 4 q_1$) that
\[ \sum_{T \neq 0} y_T^2 \otimes y_T^*(\xi_i \xi_j) + a_0 q_1 \otimes y_0^*(\xi_i \xi_j)
     = \begin{cases}
          0              & \text{if $i, j < 8$;} \\
          \half \delta_i & \text{if $i < j = 8$;} \\
          \delta_8       & \text{if $i = j = 8$}.
       \end{cases}
\]
Evaluating at~$y_0 = 2(\xi_1 \xi_8 - \xi_2 \xi_7 + \xi_3 \xi_6 - \xi_4 \xi_5)$, we find that
\[ a_0 q_1 = \delta_1 = 4 q_1 . \]
This shows that $a_0 = 4$.
(Note that if we evaluate at~$y_T$, we recover the relation
\[ y_T^2 = \sum_{j=1}^7 \tfrac{1}{2} \delta_j \cdot [\xi_j \xi_8] y_T
             + \delta_8 \cdot [\xi_8^2] y_T
         = \sum_{j=1}^7 \eps_{9-j} \tau_{9-j} \delta_{j} + \delta_8 .)
\]
We have shown:

\begin{Lemma}
  The element $\beta \in \tilde{V}_2 \otimes \tilde{V}_2 \otimes V_2^*$ is given by
  \[ \beta = \sum_{T \neq 0} y_T \otimes y_T \otimes y_T^*
              + 4\, \Xi \otimes \Xi \otimes y_0^* .
  \]
  In terms of matrices, we have that
  \begin{equation} \label {E:B}
    2 B(\vxi, \vz)
      = \sum_{T \neq 0} \frac{y_T(\vxi) y_T(\vz)}{4 r(T)} M_T S
          + \Xi(\vxi) \Xi(\vz) S .
  \end{equation}
\end{Lemma}

To get the expression for~$B$, note that $y_T^*$ corresponds to the matrix
\[ \bigl(y_T^*(\xi_i \xi_j)\bigr)_{i,j}
     = \frac{1}{r(T)} \bigl(\langle z_T, \xi_i \xi_j \rangle\bigr)_{i,j}
     = \frac{1}{8 r(T)} M_T S .
\]
The resulting matrix of bi-quadratic forms corresponding to the first summand
in~\eqref{E:B} has entries that can
be written as elements of $\Z[f_0,\ldots,f_8][\vxi,\vz]$.
The entries are given in the file \texttt{Kum3-biquforms.magma} at~\cite{Files}.
More precisely, let
\[ q = \xi_1 (f_3 f_5 \xi_4 + f_1 f_7 \xi_5)
        + f_1 \xi_2 \xi_3 + f_3 \xi_2 \xi_6 + f_5 \xi_3 \xi_7 + f_7 \xi_6 \xi_7
        + (\xi_4 + \xi_5) \xi_8 ,
\]
then the entries of
\[ B(\vxi, \vz) - \tfrac{1}{2} \bigl(q(\vxi) q(\vz) + \Xi(\vxi) \Xi(\vz)\bigr) S \]
are (up to addition of multiples of $y_0(\vxi)$ and $y_0(\vz)$)
in $\Z[f_0,\ldots,f_8][\vxi,\vz]$. (Note that $q \equiv \Xi \bmod (2, y_0)$
so that the term in parentheses is divisible by~2.)

We can now use the matrix~$B$ to perform `pseudo-addition' on~$\CK$ in complete
analogy to the case of genus~$2$ described in~\cite{FlynnSmart}. This means
that given $\kappa(P)$, $\kappa(Q)$ and~$\kappa(P-Q)$, we can find~$\kappa(P+Q)$.
This in turn can be used to compute multiples of points on~$\CK$ by a variant
of the usual divide-and-conquer scheme (`repeated squaring').

We can make the upper left entry of~$B$ completely explicit.

\begin{Lemma}
  Recall that $\langle \cdot, \cdot \rangle_S$ denotes the bilinear form corresponding
  to the matrix~$S$. We have that
  \[ B_{11}(\vxi,\vz) \equiv \langle \vxi, \vz \rangle_S^2 \bmod (y_0(\vxi), y_0(\vz)) . \]
\end{Lemma}

\begin{proof}
  This follows from $\langle z_T, \xi_1^2  \rangle = [\xi_8^2] y_T = 1$
  (for $T \neq 0$) and Corollary~\ref{C:sumyT2}:
  \begin{align*}
    B_{11}(\vxi,\vz)
      &\equiv \sum_{T \neq 0} \frac{y_T(\vxi) y_T(\vz)}{8 r(T)}
       = \langle \vxi, \vz \rangle_S^2 . \qedhere
  \end{align*}
\end{proof}

\begin{Corollary}
  For two points $P, Q \in \CJ$ with images $\kappa(P), \kappa(Q) \in \CK$, we have that
  \[ P \pm Q \in \Theta \iff \langle \kappa(P), \kappa(Q) \rangle_S = 0 . \]
\end{Corollary}

\begin{proof}
  The bilinear form associated to~$S$ vanishes if and only
  if $B_{11}(\kappa(P),\kappa(Q))$ vanishes,
  which means that $\xi_1(P+Q) \xi_1(P-Q) = 0$, which in turn is equivalent
  to $P + Q \in \Theta$ or $P - Q \in \Theta$.
\end{proof}

This is analogous to the duality between the Kummer Surface and the Dual Kummer Surface
in the case of a curve of genus~$2$, see~\cite[Thm.~4.3.1]{CasselsFlynn}.
The difference is that here the Kummer variety is self-dual.

We can now also describe the locus of vanishing of~$y_T$ on~$\CK$.

\begin{Corollary}
  Let $T \neq 0$ be an even $2$-torsion point. Then for $P \in \CJ$, we have
  that $y_T(\kappa(P)) = 0$ if and only if $2 P + T \in \Theta$.
\end{Corollary}

\begin{proof}
  This is because $y_T^2 = \langle \kappa(T), \vd \rangle_S$ (up to scaling).
\end{proof}

For $T = 0$, we get that $\Xi(\kappa(P)) = 0$ if and only if $2P \in \Theta$.
This is because $4 \Xi^2 = \delta_1$.


\section{Further properties of the duplication \\ and the sum-and-difference maps} \label{S:further}

With a view of considering bad reduction later, we now allow $k$ to be any field
and $F \in k[x,z]$ to be any binary form of degree~$8$; in particular, $F = 0$ is allowed.
Note that the relations deduced so far are valid
over~$\Z[f_0,\ldots,f_8]$ and so can be specialized to any $k$ and~$F$.
In this context, $\CK$ denotes the variety in~$\BP^7_k$ defined by
the specializations of the quadric and the $34$~quartics that define the
Kummer variety in the generic case, and $\delta$ denotes the rational map
(which now may have base points) from~$\CK$ to itself given by the quartics~$\vd$.
We can also still consider factorizations $F = G H$ into two
factors of degree~$4$ (if $F = 0$, we take both of the factors to be the zero form
of degree $4$) and obtain points on~$\CK$ that are specializations
of the images of $2$-torsion points. We will call equivalence classes of
such factorizations (up to scaling) `nontrivial even $2$-torsion points'
for simplicity, even though they do not in general arise from points of
order~$2$ on some algebraic group. If $T$ is such a nontrivial even $2$-torsion
point, then we denote the corresponding point on~$\CK$ by~$\kappa(T)$.
We normalize the coordinates of~$\kappa(T)$ such that the first coordinate is~$1$.
We also have the associated quadratic form~$y_T$. If $F = 0$, we obtain for example
$\kappa(T) = (1 : 0 : \ldots : 0)$ for the unique nontrivial even $2$-torsion point,
with associated quadratic form $y_T = \xi_8^2$.

We now state explicit criteria for the vanishing of~$\vd$ at a point on~$\CK$.
We first exhibit a necessary condition. For the following, we assume $k$ to be
algebraically closed and of characteristic $\neq 2$.

\begin{Remark}
  Note that in characteristic~$2$ we have that
  $\delta_1 = \ldots = \delta_7 = 0$ and $\delta_8 = y_T^2$ on~$\CK$ for all~$T$,
  where
  \[ y_T = \xi_8^2 + f_6 f_8 \xi_7^2 + f_4 f_8 \xi_6^2 + f_2 f_8 \xi_5^2 + f_4 f_6 \xi_4^2
            + f_2 f_6 \xi_3^2 + f_2 f_4 \xi_2^2 + f_2 f_4 f_6 f_8 \xi_1^2 ,
  \]
  which is the square of a linear form over~$k$ when $k$ is perfect.
  Let $\CL$ denote the hyperplane defined by this linear form. Then $\delta$
  restricts to a morphism on~$\CK \setminus \CL$,
  which is constant with image the origin~$(0 : \ldots : 0 : 1)$.
\end{Remark}

Assume for now that $F \neq 0$ and write
\[ F = F_0^2 F_1 \qquad \text{with $F_1$ squarefree.} \]
We define $\CT(F)$ to be the set of nontrivial even $2$-torsion points~$T$ associated
to factorizations $(G, H)$ with $G$ and~$H$ both divisible by~$F_0$.
So $\CT(F)$ is in bijection with the unordered partitions of the roots of~$F_1$
into two sets of equal size.
We also define $\CT(0)$ to be the one-element set $\{T\}$, where $T$ corresponds
to the factorization $0 = 0 \cdot 0$.

\begin{Lemma} \label{L:inherit}
  With the notation introduced above, the following statements are equivalent
  for a point on~$\CK$ with coordinate vector~$\vxi$:
  \begin{enumerate}[\upshape(i)]
    \item For all $T \in \CT(F)$, we have that $\langle \kappa(T), \vd(\vxi) \rangle_S = 0$.
    \item For all $T \in \CT(F)$, we have that $\langle \kappa(T), \vxi \rangle_S = 0$.
  \end{enumerate}
  In particular, $\vd(\vxi) = 0$ implies that $\langle \kappa(T), \vxi \rangle_S = 0$
  for all $T \in \CT(F)$.
\end{Lemma}

\begin{proof}
  By Theorem~\ref{Thm:delta}~(3), we have for all $T \in \CT(F)$ that
  $y_T(\vxi)^2 = \langle \kappa(T), \vd(\vxi) \rangle_S$, so (i) is equivalent
  to $y_T(\vxi) = 0$ for all $T \in \CT(F)$.
  When $F = 0$, we have $y_T = \xi_8^2$ and $\kappa(T) = (1:0:\ldots:0)$ for
  the unique $T \in \CT(F)$, so $y_T(\vxi) = 0$ is equivalent to $\xi_8 = 0$,
  which is equivalent to $\langle \kappa(T), \vxi \rangle_S = 0$.
  If, at the other extreme, $F$ is squarefree, then one checks\star\ that
  the coordinate vectors of the points in~$\CT(F)$ are linearly independent,
  which implies that (i) is equivalent to $\vd(\vxi) = 0$ and (ii) is
  equivalent to $\vxi = 0$. The claim then follows from
  Theorem~\ref{Thm:delta}~(4).

  We now assume that $F \neq 0$ and write $F = F_0^2 F_1$ as above with $F_1$ squarefree
  and $F_0$ non-constant. We check by an explicit computation\star\ that

  (*) \emph{the $y_T$ for $T \in \CT(F)$ form a basis of the symmetric square
  of the space spanned by the linear forms $\langle \kappa(T), \cdot \rangle_S$
  for $T \in \CT(F)$.}

  This implies that the vanishing of the~$y_T$ is equivalent to~(ii).
  To verify~(*), we can apply a transformation moving the
  roots of~$F_0$ to an initial segment of~$(0, \infty, 1, a)$
  (where $a \in k \setminus \{0,1\}$). The most involved case is when $\deg F_0 = 1$.
  We can then take $F_0 = x$ and find that the linear forms given by the
  $T \in \CT(F)$ span $\langle \xi_4, \xi_6, \xi_7, \xi_8 \rangle$ and that
  the $10 \times 10$ matrix whose rows are the coefficient vectors of the~$y_T$
  with respect to the monomials of degree~$2$ in these four variables
  has determinant a power of two times a power of $\disc(F_1)$, hence is invertible.
  The other cases are similar, but simpler.
\end{proof}

This prompts the following definition.

\begin{Definition}
  We write $\CK_\good$ for the open subscheme
  \[ \CK \setminus \{P : \text{$\langle \kappa(T), P \rangle_S = 0$ for all $T \in \CT(F)$}\} \]
  of~$\CK$.
\end{Definition}

Lemma~\ref{L:inherit} now immediately implies the following.

\begin{Corollary} \label{C:closed2}
  The rational map $\delta$ on~$\CK$ restricts to a
  morphism $\CK_\good \to \CK_\good$.
\end{Corollary}

We will now consider the `bad' subset $\CK \setminus \CK_\good$ of~$\CK$
in more detail, in particular in relation to the base locus of~$\delta$,
which it contains according to Corollary~\ref{C:closed2}.
We begin with a simple sufficient condition for a point to be in the base locus.

\begin{Lemma} \label{L:d=0suff}
  Assume that $F(x,z)$ is divisible by~$z^2$.
  Let $\vxi$ be the coordinate vector of a point on~$\CK$ such that
  $\xi_2 = \xi_3 = \xi_4 = \xi_8 = 0$. Then $\vd(\vxi) = 0$.
\end{Lemma}

\begin{proof}
  Plugging $f_7 = f_8 = \xi_2 = \xi_3 = \xi_4 = \xi_8 = 0$ into the expressions
  for the~$\delta_j$ gives zero\star.
\end{proof}

We set
\[ \CL_\infty = \{(\xi_1 : \ldots : \xi_8) \in \BP^7 : \xi_2 = \xi_3 = \xi_4 = \xi_8 = 0\} . \]
Using the formulas given in Section~\ref{S:trans} for the action
on~$\vxi$, one sees easily that $\CL_\infty$ is invariant under scaling of~$x$
and also under shifting $(x, z) \mapsto (x + \lambda z, z)$ (always assuming that $f_7 = f_8 = 0$),
which together generate the stabilizer of~$\infty$ in~$\PGL(2)$.

For $F$ with a multiple root at some point $a \in \BP^1$, let $\tilde{F}$ be the
result of acting on~$F$ by a linear substitution~$\phi$ that moves $a$ to~$\infty$;
then $\tilde{F}$ is divisible by~$z^2$. We write $\CL_a \subset \BP^7$ for the image
of~$\CL_\infty$ under the automorphism of~$\BP^7$ induced by~$\phi^{-1}$.
Since the stabilizer of~$\infty$ in~$\PGL(2)$ leaves $\CL_\infty$ invariant,
this definition of~$\CL_a$ does not depend on the choice of~$\phi$.
For example,
\[ \CL_0 = \{(\xi_1 : \ldots : \xi_8) \in \BP^7 : \xi_4 = \xi_6 = \xi_7 = \xi_8 = 0\} . \]
We write $A(F) \subset \BP^1$ for the set of multiple roots of~$F$.
This is all of~$\BP^1$ when $F = 0$. Otherwise, $A(F)$ consists of the roots
of~$F_0$ when $F = F_0^2 F_1$ with $F_1$ squarefree.

\begin{Corollary} \label{C:d=0suff}
  If $P \in \CK \cap \CL_a$ for some $a \in A(F)$, then $\vd(P) = 0$.
\end{Corollary}

\begin{proof}
  This follows from Lemma~\ref{L:d=0suff} by applying a suitable automorphism of~$\BP^1$.
\end{proof}

So the base locus of~$\delta$ contains $\CK \cap \bigcup_{a \in A(F)} \CL_a$.
When $F$ is not a nonzero square, we can show that this is exactly
the `bad set' $\CK \setminus \CK_\good$.

\begin{Lemma} \label{L:d=0nec1}
  Assume that $F$ is not of the form $F = H^2$ with $H \neq 0$.
  Let $P$ be in the `bad set' $\CK \setminus \CK_\good$.
  Then $P \in \CL_a$ for some $a \in A(F)$. In particular,
  \[ \CK_\good = \CK \setminus \bigcup_{a \in A(F)} \CL_a , \]
  and $\CK \setminus \CK_\good = \CK \cap \bigcup_{a \in A(F)} \CL_a$
  is the base locus of~$\delta$.
\end{Lemma}

\begin{proof}
  Let $\vxi$ be a coordinate vector for~$P$. We write $F = F_0^2 F_1$ with $F_1$
  squarefree.
  We split the proof into various cases according to the factorization type of~$F_0$.
  If $F_0$ is constant, there is nothing to prove.
  Otherwise we move the roots of~$F_0$ to an initial segment of $(0, \infty, 1)$.
  \begin{enumerate}[1.]
    \item $F_0 = x$. In this case the assumption is equivalent to
          $\xi_4 = \xi_6 = \xi_7 = \xi_8 = 0$ (compare the proof of Lemma~\ref{L:inherit}),
          so that $P \in \CL_0$.
    \item $F_0 = x^2$. The assumption is that $\xi_7 = \xi_8 = 0$; using the equations
          defining~$\CK$ this implies\star\ that $\xi_4 = \xi_6 = 0$, so $P \in \CL_0$.
    \item $F_0 = x^3$. The assumption is that $\xi_8 = 0$,
          which implies\star\ that $\xi_7 = \xi_6 = \xi_4 = 0$, so $P \in \CL_0$.
    \item $F_0 = x z$. In this case the assumption is that $\xi_4 = \xi_8 = 0$,
          which then implies\star\ that $\xi_6 = \xi_7 = 0$
          or $\xi_2 = \xi_3 = 0$, and so $P \in \CL_0$ or~$P \in \CL_\infty$.
    \item $F_0 = x^2 z$. The assumption is that $\xi_8 = 0$,
          which leads to\star\ $P \in \CL_0$ or~$P \in \CL_\infty$.
    \item $F_0 = x z (x-z)$. A similar computation
          shows\star\ that $P \in \CL_0 \cup \CL_1 \cup \CL_\infty$.
    \item $F = 0$. Here the assumption is that $\xi_8 = 0$.
          The intersection $\CK \cap \{\xi_8 = 0\}$ is defined\star\ by the $2 \times 2$-minors
          of the matrix
          \[ \begin{pmatrix} \xi_2 & \xi_3 & \xi_4 \\
                             \xi_3 & \xi_4 + \xi_5 & \xi_6 \\
                             \xi_4 & \xi_6 & \xi_7
             \end{pmatrix} ,
          \]
          which therefore has rank~$1$ when evaluated on any point in $\CK \cap \{\xi_8 = 0\}$.
          If $\xi_2 = 0$, then this implies that
          $\xi_3 = \xi_4 = 0$ as well, so that $P \in \CL_\infty$. Otherwise,
          we can make a transformation shifting $x/z$ by~$\lambda$ as in Section~\ref{S:trans}
          that makes $\tilde{\xi}_7 = 0$ ($\tilde{\xi}_7$ is a polynomial of degree~$4$ in~$\lambda$
          with leading coefficient~$\xi_2$, so we can find a suitable~$\lambda$,
          since $k$ is assumed to be algebraically closed). Then we get that
          $\tilde{\xi}_8 = \tilde{\xi}_7 = \tilde{\xi}_6 = \tilde{\xi}_4 = 0$,
          so the image point is in~$\CL_0$, hence $P \in \CL_\lambda$.
  \end{enumerate}
  The last statement follows, since
  Corollary~\ref{C:closed2} shows that the base scheme of~$\delta$ is contained
  in $\CK \setminus \CK_\good$ and Corollary~\ref{C:d=0suff} shows that
  it contains the intersection of~$\CK$ with the union of the~$\CL_a$.
\end{proof}

We now consider the case $F = F_0^2 \neq 0$. Then the curve $y^2 = F(x,z) = F_0(x,z)^2$
splits into the two components $y = \pm F_0(x,z)$. The points on~$\CK$ correspond
to linear equivalence classes of effective divisors of degree~$4$, modulo the
action of the hyperelliptic involution. So there are three distinct possibilities
how the points can be distributed among the two components: two on each,
one and three, or all four on the same component. In the last case, we have
$B \equiv \pm F_0 \bmod A$, and we can change the representative so that $B = \pm F_0$,
which makes $C = 0$. So the two components of~$\Pic^4(\CC)$ consisting of classes
of divisors whose support is contained in one of the two components of~$\CC$
map to a single point $\omega \in \CK$, which one can check\star\ coincides with $\kappa(T)$
for the single $T \in \CT(F)$; it satisfies $\vd(\omega) = 0$.

Now a point~$P$ on the component of~$\CK$ corresponding to
the distribution of one and three points on the two components, if it is not
in the base scheme of~$\delta$, must satisfy $\delta(P) = \omega$. So for such
points we have $\vd(\delta(P)) = 0$, but $\vd(P) \neq 0$.
Let $\vxi$ be coordinates for a point~$P$ with $\delta(P) = \omega = \kappa(T)$.
Then $\langle \kappa(T), \vd(\vxi) \rangle_S = \langle \kappa(T), \kappa(T) \rangle_S = 0$
(all points on~$\CK$ satisfy $\langle \vxi, \vxi \rangle_S = y_0(\vxi) = 0$).
By Lemma~\ref{L:inherit}, this is equivalent to $\langle \kappa(T), \vxi \rangle_S = 0$.
We write $\CE$ for the hyperplane given by $\langle \kappa(T), \vxi \rangle_S = 0$.
So in this case $\CK_\good = \CK \setminus \CE$, and
$P \in \CK \cap \CE = \CK \setminus \CK_\good$ does not necessarily imply that $\vd(P) = 0$.
But we still have the following.

\begin{Lemma} \label{L:d=0nec2}
  Assume that $F = F_0^2$ with $F_0 \neq 0$. If $P \in \CK$ with $\vd(P) = 0$,
  then $P \in \CL_a$ for some $a \in A(F)$ (which here is simply the set of roots
  of~$F_0$).
\end{Lemma}

\begin{proof}
  We can again assume that the roots of~$F_0$ are given by an initial segment
  of~$(0, \infty, 1, a)$ (with $a \neq \infty, 0, 1$). We consider the
  various factorization types of~$F_0$ in turn; they are represented by
  \[ F_0 = x^4,\quad x^3 z,\quad x^2 z^2,\quad x^2 z (x-z) \quad\text{and}\quad x z (x-z) (x - az) . \]
  The computations\star\ are similar
  to those done in the proof of Lemma~\ref{L:d=0nec1}. The most involved case
  is when $F_0$ has four distinct roots. To deal with it successfully, we make
  use of the Klein Four Group of automorphisms of the set of roots of~$F_0$.
\end{proof}

We now have a precise description of the base scheme of the duplication map~$\delta$
on~$\CK$, which is given by the quartic forms~$\vd$.

\begin{Proposition} \label{P:d=0crit}
  Let $k$ be an algebraically closed field of characteristic $\neq 2$
  and let $F \in k[x,z]$ be homogeneous of degree~8. We denote by $\CK$
  and~$\vd$ the objects associated to~$F$.
  \begin{enumerate}[\upshape(1)]
    \item The base locus of~$\delta$ is $\CK \cap \bigcup_{a \in A(F)} \CL_a$.
    \item The base locus of $\delta \circ \delta$ is~$\CK \setminus \CK_\good$;
          $\delta$ can be iterated indefinitely on~$\CK_\good$.
    \item If $F$ is not of the form $F = F_0^2$ with $F_0 \neq 0$, then
          the base locus of~$\delta$ is~$\CK \setminus \CK_\good$.
  \end{enumerate}
\end{Proposition}

\begin{proof} \strut
  \begin{enumerate}[(1)]
    \item Corollary~\ref{C:d=0suff} shows that the condition is sufficient.
          Conversely, if $\vd(P) = 0$, then Lemmas \ref{L:inherit}, \ref{L:d=0nec1}
          and~\ref{L:d=0nec2} show that $P \in \CL_a$ for some multiple root~$a$ of~$F$.
    \item The second statement is Corollary~\ref{C:closed2}.
          In view of~(3), it is sufficient to consider the case $F = F_0^2 \neq 0$
          for the first statement. If $P \in \CK \setminus \CK_\good$ is not in
          the base locus of~$\delta$, then $\delta(P) = \omega$, which is in
          the base locus of~$\delta$, so $P$ is in the base locus of~$\delta \circ \delta$.
          Conversely, if $P$ is in the base locus of~$\delta \circ \delta$,
          then $P$ cannot be in~$\CK_\good$ by the second statement.
    \item This follows from Corollary \ref{C:closed2} and Lemma~\ref{L:d=0nec1}.
    \qedhere
  \end{enumerate}
\end{proof}

We can state a property of the `add-and-subtract' morphism
that is similar to that of~$\delta$ given in Corollary~\ref{C:closed2}.
We write $\alpha \colon \Sym^2\CK \to \Sym^2\CK$ for the map given by
the matrix~$B$ as defined in Section~\ref{S:sumdiff};
this is defined for arbitrary~$F \in k[x,z]$, homogeneous
of degree~8. In general~$\alpha$ is only a rational map.

\begin{Lemma} \label{L:closed+}
  Let $k$ be an algebraically closed field of characteristic $\neq 2$
  and let $F \in k[x,z]$ be homogeneous of degree~8. We denote by $\CK$
  and~$\vd$ the objects associated to~$F$.
  Then $\alpha$ restricts to a morphism $\Sym^2\CK_\good \to \Sym^2\CK_\good$.
\end{Lemma}

\begin{proof}
  Note that generically, $\alpha \circ \alpha = \Sym^2 \delta$; this comes from
  the fact that
  \[ \{(P+Q)+(P-Q), (P+Q)-(P-Q)\} = \{2P, 2Q\} . \]
  If we write $\vxi \ast \vxi'$
  for the symmetric matrix $\vxi^\top\! \cdot \vxi' + {\vxi'}^\top\! \cdot \vxi$,
  then this relation shows that
  \begin{equation} \label{E:Bdimpl}
    \vz \ast \vz' = 2 B(\vxi, \vxi')
      \mathrel{\quad\Longrightarrow\quad} \vd(\vxi) \ast \vd(\vxi') = 2 B(\vz, \vz') ,
  \end{equation}
  up to a scalar factor,
  which we find to be~$1$ by taking $\vxi = \vxi' = (0,\ldots,0,1)$.
  This is then a relation that is valid over~$\Z[f_0,\ldots,f_8]$.

  Now let $\vxi$ and~$\vxi'$ be projective coordinate vectors of points in~$\CK_\good$
  and write $2 B(\vxi, \vxi') = \vz \ast \vz'$ for suitable vectors~$\vz,\vz'$.
  Then by Corollary~\ref{C:closed2}, $\vd(\vxi)$ and~$\vd(\vxi')$ both do
  not vanish, so $\vd(\vxi) \ast \vd(\vxi') \neq 0$. This implies that
  $\vz, \vz' \neq 0$, which shows that $\alpha$ is defined on~$\CK_\good$.
  If the point given by~$\vz \ast \vz'$ were not in~$\Sym^2\CK_\good$,
  then iterating $\alpha$ at most four more times would produce zero
  by Proposition~\ref{P:d=0crit}~(2), contradicting the fact that $\delta$ can be
  iterated indefinitely on the points represented by~$\vxi$ and~$\vxi'$.
\end{proof}


\section{Heights} \label{S:heights}

We now take $k$ to be a number field (or some other field
of characteristic~$\neq 2$ with a collection of absolute values satisfying
the product formula, for example a function field in one variable).
We also assume again that $F \in k[x,z]$ is a squarefree binary octic form.
Then $\CC$ is a curve of genus~$3$ over~$k$, and we have the Jacobian~$\CJ$
and the Kummer variety~$\CK$ associated to~$\CC$.
We define the \emph{naive height} on $\CJ$ and on~$\CK$ to be the standard
height on~$\BP^7$ with respect
to the coordinates $(\xi_1 : \ldots : \xi_8)$. We denote it by
\[ h(P) = \sum_v n_v \log \max\{|\xi_1(P)|_v, \ldots, |\xi_8(P)|_v\}
     \qquad \text{for $P \in \CJ(k)$ or $\CK(k)$,}
\]
where $v$ runs through the places of~$k$, the absolute values $| \cdot |_v$
extend the standard absolute values on~$\Q$ and $n_v = [K_v : \Q_w]$,
where $w$ is the place of~$\Q$ lying below~$v$, so that we have
the product formula
\[ \prod_v |\alpha|_v^{n_v} = 1 \qquad\text{for all $\alpha \in k^\times$.} \]

Then by general theory (see for example \cite[Part~B]{HindrySilverman}) the limit
\[ \hat{h}(P) = \lim_{n \to \infty} \frac{h(nP)}{n^2} \]
exists and differs from~$h(P)$ by a bounded amount. This is the
\emph{canonical height} of~$P$. One of our goals in
this section will be to find an explicit bound for
\[ \beta = \sup_{P \in \CJ(k)} \bigl(h(P) - \hat{h}(P)\bigr) . \]
We refer to~\cite{MueSto} for a detailed study of heights in the
case of Jacobians of curves of genus~$2$, with input from~\cite{StollH1}
and~\cite{StollH2}. We will now proceed to obtain some comparable
results in our case of hyperelliptic genus~$3$ Jacobians.
Most of this is based on the following telescoping series trick
going back to Tate: we write
\[ \hat{h}(P)
    = \lim_{n \to \infty} 4^{-n} h(2^n P)
    = h(P) + \sum_{n=0}^\infty 4^{-(n+1)} \bigl(h(2^{n+1} P) - 4 h(2^n P)\bigr)
\]
and split the term $h(2 P) - 4 h(P)$ into local components as follows:
\[ h(2 P) - 4 h(P)
    = \sum_v n_v \bigl(\max_j \log |\delta_j(\vxi(P))|_v
                     - 4 \max_j \log |\xi_j(P)|_v\bigr)
    = \sum_v n_v \eps_v(P)
\]
with
$\eps_v(P) = \max_j \log |\delta_j(\vxi(P))|_v - 4 \max_j \log |\xi_j(P)|_v$,
which is independent of the scaling of the coordinates~$\vxi(P)$ and so
can be defined for all $P \in \CJ(k_v)$ or $\CK(k_v)$.
Then $\eps_v \colon \CK(k_v) \to \R$ is continuous, so (since $\CK(k_v)$ is compact)
it is bounded.
If $-\gamma_v \le \inf_{P \in \CK(k_v)} \eps_v(P)$, then we have that
\[ \beta \le \sum_v n_v \sum_{n=0}^\infty 4^{-(n+1)} \gamma_v
         = \tfrac{1}{3} \sum_v n_v \gamma_v .
\]
So we will now obtain estimates for~$\gamma_v$. We follow closely the
strategy of~\cite{StollH1}.
Note that writing
\[ \mu_v(P) = \sum_{n=0}^\infty 4^{-(n+1)} \eps_v(2^n P)
            = \lim_{n \to \infty} 4^{-n} \max_j \log |\vd^{\circ n}(\vxi(P))|_v
                - \max_j \log |\xi_j(P)| ,
\]
we also have that
\[ \hat{h}(P) = h(P) + \sum_v n_v \mu_v(P) . \]

We assume that the polynomial defining the curve~$\CC$ has coefficients in
the ring of integers of~$k$. Then the matrices~$M_T$ defined in Section~\ref{S:2tors}
for even $2$-torsion points have entries that are algebraic integers.
We use $\CO$ to denote the ring of all algebraic integers.
Let $\vxi$ be coordinates of a point on~$\CK$.
Then Theorem~\ref{Thm:delta}~(3) tells us that for all even $2$-torsion
points $T \neq 0$, we have that
\[ y_T(\vxi)^2 \in \CO \delta_1(\vxi) + \CO \delta_2(\vxi) + \ldots + \CO \delta_8(\vxi) \]
and Lemma~\ref{L:x2iny} tells us that
(note that the coefficient of $\xi_{9-j}^2$ in~$y_0$ is zero)
\[ \xi_j^2 \in \sum_{T \neq 0, \text{even}} \frac{1}{8 r(T)} \CO y_T(\vxi) . \]

\begin{Lemma} \label{L:nonarch}
  Let $v$ be a non-archimedean place of~$k$. Then for $P \in \CK(k_v)$,
  we have that
  \[ \log |2^6 \disc(F)|_v \le \log \min_T |2^6 r(T)^2|_v \le \eps_v(P) \le 0 , \]
  where $T$ runs through the non-trivial even $2$-torsion points.
\end{Lemma}

\begin{proof}
  Let $\vxi$ be coordinates for~$P$ and write
  $d_j = \delta_j(\vxi)$ for $j = 1, \ldots, 8$. Then for all even~$T \neq 0$,
  \[ |y_T(\vxi)|^2_v \le \max_j |d_j|_v \]
  and
  \[ |\xi_j|^4_v \le \max_T |8 r(T)|^{-2}_v |y_T(\vxi)|^2_v
                 \le \max_T |8 r(T)|^{-2}_v \max_j |d_j|_v .
  \]
  So
  \[ \eps_v(P) = \log \max_j |d_j|_v - 4 \log \max_j |\xi_j|_v
               \ge \log \min_T |2^6 r(T)^2|_v .
  \]
  Since $r(T)^2$ divides the discriminant~$\disc(F)$, the first inequality
  on the left also follows. The upper bound follows from the fact that
  the polynomials~$\delta_j$ have integral coefficients.
\end{proof}

Since $\eps_v(P)$ is an integral multiple of the logarithm of the absolute
value of a uniformizer~$\pi_v$,
we can sometimes gain a little bit by using
\[ \eps_v(P) \ge
    -\Bigl\lfloor \max_T v\bigl(|2^6 r(T)^2|\bigr) \Bigr\rfloor \log |\pi_v|_v ,
\]
where $v$ denotes the $v$-adic additive valuation, normalized so that $v(\pi_v) = 1$.

\begin{Example} \label{Ex:C1two}
  For the curve
  \[ y^2 = 4 x^7 - 4 x + 1 \]
  over~$\Q$
  and $v = 2$, the discriminant bound gives\star\ $\eps_2(P) \ge -22 \log 2$, since the
  discriminant of the polynomial on the right hand side (considered as a dehomogenized
  binary octic form) has 2-adic valuation~$16$. To get a better bound, we consider
  the resultants~$r(T)$. If we write
  \[ f(x) = 4 x^7 - 4 x + 1 = 4 g(x) h(x) \]
  with $g$ and~$h$ monic of degree $3$ and~$4$, respectively, then
  $r(T) = 2^8 \Res(g, h)$. From the Newton Polygon of~$f$ we see that all
  roots~$\theta$ of~$f$ satisfy $v_2(\theta) = -2/7$. This gives $v_2(r(T)) \ge 32/7$.
  Since the product of all $35$ resultants~$r(T)$ is the tenth power of the discriminant,
  we must have equality. This gives the bound
  $\eps_2(P) \ge -(15 + \frac{1}{7}) \log 2$,
  which can be improved to $-15 \log 2$, so that we get $-\mu_2 \le 5 \log 2$.
\end{Example}

\begin{Corollary}
  Assume that $k = \Q$. Then we have that
  \[ \beta \le \tfrac{1}{3} \log |2^6 \disc(F)| + \tfrac{1}{3} \gamma_\infty . \]
\end{Corollary}

To get a bound on~$\gamma_\infty$, we use the archimedean triangle inequality.
We write $\tau_j(T)$ for the coordinates of a non-trivial even $2$-torsion
point~$T$ (with $\tau_1(T) = 1$) and $\upsilon_j(T)$ for the coefficients
in the formula for~$\xi_j^2$, so that we have
\[ \xi_j^2 = \sum_T \upsilon_j(T) y_T . \]

\begin{Lemma} \label{L:arch}
  Let $v$ be an archimedean place of~$k$. Then we have that
  \[ \gamma_v \le \log \max_j \left(\sum_T |\upsilon_j(T)|_v
                                     \sqrt{\sum_{i=1}^8 |\tau_i(T)|_v} \right)^2 .
  \]
\end{Lemma}

\begin{proof}
  Similarly as in the non-archimedean case, we have that
  \[ |y_T(\vxi)|^2_v \le \sum_{j=1}^8 |\tau_j(T)|_v \max_j |d_j|_v \]
  and
  \[ \max_j |\xi_j|^2_v
       \le \max_j \sum_T |\upsilon_j(T)|_v |y_T(\vxi)|_v .
  \]
  Combining these gives the result.
\end{proof}

As in~\cite[Section~16B]{MueSto}, we can refine this result somewhat. Define a function
\[ f \colon \R_{\ge 0}^8 \To \R_{\ge 0}^8, \quad
       (d_1, \ldots, d_8) \longmapsto
           \left(\sqrt{\sum_T |\upsilon_j(T)|_v
                     \sqrt{\sum_{i=1}^8 |\tau_i(T) d_{9-i}|_v}}\right)_{1 \le j \le 8} .
\]
We write $\|(x_1,\ldots,x_8)\|_\infty = \max\{|x_1|, \ldots, |x_8|\}$ for the maximum norm.

\begin{Lemma} \label{L:arch1}
  Define a sequence $(b_n)$ in~$\R_{\ge 0}^8$ by
  \[ b_0 = (1, \ldots, 1) \qquad\text{and}\qquad b_{n+1} = f(b_n) . \]
  The $(b_n)$ converges to a limit~$b$, and we have that
  \[ -\mu_v(P) \le \frac{4^N}{4^N-1} \log \|b_{N}\|_\infty  \]
  for all $N \ge 1$ and all $P \in \CJ(\C)$. In particular, $\sup -\mu_v(\CJ(\C)) \le \log \|b\|_\infty$.
\end{Lemma}

\begin{proof}
  See the proof of~\cite[Lemma~16.1]{MueSto}.
\end{proof}

\begin{Example} \label{Ex:C1infty}
  For the curve
  \[ y^2 = 4 x^7 - 4 x + 1 , \]
  the bound $\gamma_\infty/3$ is $1.15134$, whereas with $N = 8$, we obtain the
  considerably better bound $-\mu_\infty \le 0.51852$.

  We can improve this a little bit more if $k_v = \R$, by making use of the fact that
  the coordinates of the points involved are real, but the $\tau_i(T)$ may be non-real.
  This can give a better bound on
  \[ |y_T^2|_v
      \le \max_{|\delta_i| \le d_i} \left| \sum_{i=1}^8 \eps_i \tau_i(T) \delta_{9-i} \right|_v
      .
  \]
  For the curve above, this improves\star\ the upper bound for~$-\mu_\infty$ to~$0.43829$.
\end{Example}

Now we show that in the most common cases of bad reduction, there is in fact
no contribution to the height difference bound.
This result is similar to~\cite[Proposition~5.2]{StollH2}.

\begin{Lemma} \label{L:regular}
  Let $v$ be a non-archimedean place of~$k$ of odd residue characteristic.
  Assume that the reduction of~$F$ at~$v$ has a simple root
  and that the model of~$\CC$ given by $y^2 = F(x,z)$ is regular at~$v$.
  Then $\mu_v(P) = \eps_v(P) = 0$ for all~$P \in \CJ(k_v)$.
\end{Lemma}

Note that the assumptions on the model are satisfied when $v(\disc(F)) = 1$.

\begin{proof}
  We work with a suitable unramified extension~$K$ of~$k_v$, so that
  the reduction~$\bar{F}$ of~$F$ splits into linear factors over the residue field.
  We denote the ring of integers of~$K$ by~$\CO$.
  By assumption, $\bar{F}$ has a simple root, which by Hensel's Lemma lifts
  to a root of~$F$ in~$\BP^1(K)$. We can use a transformation
  defined over~$\CO$ to move this root of~$F$ to~$\infty$.
  Then we have $f_8 = 0$ and $v(f_7) = 0$.
  We can further scale~$F$ (at the cost of at most a further quadratic unramified extension)
  so that $f_7 = 1$.

  Assume that $P \in \CJ(K)$ has $\eps_v(P) \neq 0$ and let $\vxi$ be
  normalized coordinates for $\kappa(P) \in \CK(K)$ (i.e., such that the
  coordinates are in~$\CO$ and at least one of them is in~$\CO^\times$).
  By Proposition~\ref{P:d=0crit}, the reduction of~$P$ must lie in some~$\CL_a$
  where $a \neq \infty$ is a multiple root of~$\bar{F}$. We can shift $a$ to~$0$;
  then the coordinates $\xi_4$, $\xi_6$, $\xi_7$ and~$\xi_8$ have positive valuation.
  We also have $v(f_0) = 1$
  (this is because the model is regular at the point $(0:0:1)$ in the reduction)
  and $v(f_1) \ge 1$ (since $a = 0$ is a multiple root of~$\bar{F}$).

  Now assume first that $v(\xi_1) = 0$; then we can scale~$\vxi$ such that $\xi_1 = 1$.
  We consider the quantity~$\mu_{034}$ introduced in Section~\ref{S:lift};
  its value on~$P$ is in~$K$.
  By~\eqref{E:muijk}, we have that
  \[ \mu_{034}^2 = \eta_{00} \eta_{34}^2 + \eta_{33} \eta_{04}^2 + \eta_{44} \eta_{03}^2
                   - 4 \eta_{00} \eta_{33} \eta_{44} - \eta_{03} \eta_{04} \eta_{34} \\
                 = f_0 + (f_6 - \xi_2) \xi_4^2 - \xi_6 \xi_4
  \]
  (note that $\eta_{44} = f_8 = 0$, $\eta_{34} = f_7 = 1$,
  $\eta_{33} = f_6 - \eta_{24}$, $\eta_{24} = \xi_2$, $\eta_{04} = \xi_4$,
  $\eta_{03} = \xi_6$). Now since $v(f_0) = 1$ and $v(\xi_4) \ge 1$, $v(\xi_6) \ge 1$,
  we find that $2 v(\mu_{034}) = 1$, a contradiction.

  So we must have $v(\xi_1) > 0$. One can check\star\ that
  \begin{align*}
    \nu_1 &= (\xi_4 - \xi_5) \mu_{013} + \xi_7 \mu_{123} \\
    \nu_2 &= \xi_3 \mu_{014} - \xi_4 \mu_{024} \\
    \nu_3 &= \xi_2 \mu_{024} - \xi_4 \mu_{134}
  \end{align*}
  are functions in~$L(4 \Theta)$,  which are clearly odd, so their squares can be written as
  quartics in the~$\xi_j$ by Lemma~\ref{L:Mumford_even}.
  Let $I$ be the square of the ideal generated by $f_0,f_1,\xi_1,\xi_4,\xi_6,\xi_7,\xi_8$;
  then anything in~$I$ has valuation at least~$2$. We find\star\ that modulo~$I$,
  \[ \nu_1^2 \equiv f_0 \xi_5^4, \qquad
     \nu_2^2 \equiv f_0 \xi_3^4, \qquad
     \nu_3^2 \equiv f_0 \xi_2^4 .
  \]
  Since (at least) one of $\xi_2$, $\xi_3$, $\xi_5$ is a unit and $v(f_0) = 1$,
  we obtain a contradiction again.

  Therefore $\eps_v(P) = 0$ for all $P \in \CJ(K)$, which implies that
  $\mu_v(P) = 0$ as well.
\end{proof}

\begin{Example} \label{Ex:diff}
  The discriminant of the curve
  \[ \CC \colon y^2 = 4 x^7 - 4 x + 1 \]
  is\star\ $2^{28} \cdot 19 \cdot 223 \cdot 44909$. Lemma~\ref{L:regular} now implies
  that $\eps_v(P) = 0$ for all $P \in \CJ(\Q_v)$ for all places~$v$ except
  $2$ and~$\infty$, including the bad primes $19$, $223$ and~$44909$.
  So, using Examples \ref{Ex:C1two} and~\ref{Ex:C1infty}, we obtain the bound
  \[ h(P) \le \hat{h}(P) + 5 \log 2 + 0.43829 \le \hat{h}(P) + 3.90403 \]
  for all $P \in \CJ(\Q)$.
\end{Example}

To compute the canonical height $\hat{h}(P)$ for some point $P \in \CJ(\Q)$
(say, for a hyperelliptic curve~$\CC$ of genus~$3$ defined over~$\Q$),
we can use any of the approaches described in~\cite{MueSto}, except the most
efficient one (building on Proposition~14.3 in loc.~cit.), since we have so far
no general bound on the denominator of $\mu_p/\log p$ in terms of the discriminant.
A little bit of care is needed, since contrary to the genus~$2$ situation,
$\eps_v = 0$ and~$\mu_v = 0$ are not necessarily equivalent ---
there can be a difference when the reduction of~$F$ is a constant times a square ---
so the criterion for a point to be in the subgroup on which $\mu_v = 0$
has to be taken as $\overline{\kappa(P)} \in \CK_\good(\F)$, where $\overline{\kappa(P)}$
is the reduction of~$\kappa(P)$ at~$v$ and $\F$ is the residue class field.

We can describe the subset on which $\mu_v = 0$ and show that it is a subgroup
and that $\mu_v$ factors through the quotient.

\begin{Theorem}
  Let $v$ be a non-archimedean place of~$k$ of odd residue characteristic.
  Write $\CJ(k_v)_\good$ for the subset of~$\CJ(k_v)$ consisting of the points~$P$
  such that $\kappa(P)$ reduces to a point in~$\CK_\good(\F)$.
  Then $\CJ(k_v)_\good = \{P \in \CJ(k_v) : \mu_v(P) = 0\}$
  is a subgroup of finite index of~$\CJ(k_v)$, and
  $\eps_v$ and~$\mu_v$ factor through the quotient $\CJ(k_v)/\CJ(k_v)_\good$.
\end{Theorem}

\begin{proof}
  That $\CJ(k_v)_\good$ is a group follows from Lemma~\ref{L:closed+}:
  If $P_1$ and~$P_2$ are in~$\CJ(k_v)_\good$, then $P_1 \pm P_2$ reduce
  to a point in~$\CK_\good$ as well. This subgroup contains the kernel
  of reduction, which is of finite index, so it is itself of finite index.
  That $\CJ(k_v)_\good = \{P \in \CJ(k_v) : \mu_v(P) = 0\}$ follows from
  the results of Section~\ref{S:further}.

  It remains to show that $\mu_v$ (and therefore also $\eps_v$, since
  $\eps_v(P) = 4 \mu_v(P) - \mu_v(2P)$) factors through the quotient group.
  Let $P, P' \in \CJ(k_v)$ and let $\vxi$ and~$\vxi'$ be coordinate vectors
  for $\kappa(P)$ and~$\kappa(P')$, respectively. We can then choose
  coordinate vectors $\vz$ and~$\vz'$ for $\kappa(P'+P)$ and~$\kappa(P'-P)$, respectively,
  such that $\vz \ast \vz' = 2 B(\vxi, \vxi')$.
  Iterating the implication in~\eqref{E:Bdimpl} then gives
  \[ \vd(\vz) \ast \vd(\vz') = 2 B\bigl(\vd(\xi), \vd(\xi')\bigl) , \]
  and we can iterate this relation further. If $\underline{\alpha}$ is a vector
  or matrix, then we write $|\underline{\alpha}|_v$ for the maximum
  of the $v$-adic absolute values of the entries of~$\alpha$.
  Define
  \[ \eps_v(P,P') = \log |2B(\vxi,\vxi')|_v - 2 \log |\vxi|_v - 2 \log |\vxi'|_v \]
  (this does not depend on the scaling of the coordinate vectors) and note
  that $|\vz \ast \vz'|_v = |\vz|_v \cdot |\vz'|_v$ (here we use that the
  residue characteristic is odd).
  We then see that $\mu_v(P) = 0$ implies $\mu_v(P+Q) = \mu_v(Q)$ for
  all~$Q \in \CJ(k_v)$ in the same way as in the proof of~\cite[Lemma~3.7]{MueSto}.
\end{proof}


\section{An application} \label{S:example}

We consider the curve
\[ \CC' \colon y^2 - y = x^7 - x , \]
which is isomorphic to the curve
\[ \CC \colon y^2 = 4 x^7 - 4 x + 1 , \]
which we have been using as our running example.
Our results can now be used to determine a set of generators for the
Mordell-Weil group~$\CJ(\Q)$. This is the key ingredient for the
method that determines the set of integral points on a hyperelliptic
curve as in~\cite{BMSST}. We carry out the necessary computations and
thence find all the integral solutions of the equation $y^2 - y = x^7 - x$.

A 2-descent on the Jacobian~$\CJ$ of~$\CC$ as described in~\cite{StollDesc}
and implemented in Magma~\cite{Magma} shows that the rank of~$\CJ(\Q)$
is at most~$4$. We have $\#\CJ(\F_3) = 94$ and $\#\CJ(\F_7) = 911$,
which implies that $\CJ(\Q)$ is torsion free (the torsion subgroup injects
into~$\CJ(\F_p)$ for $p$ an odd prime of good reduction). We have the
obvious points $(0, \pm 1)$, $(\pm 1, \pm 1)$, $(\pm \omega, \pm 1)$, $(\pm \omega^2, \pm 1)$
on~$\CC$, where $\omega$ denotes a primitive cube root of unity, together
with the point at infinity. We can check that the rational divisors of
degree zero on~$\CC$ supported in these points generate a subgroup~$G$ of~$\CJ(\Q)$
of rank~$4$, which already shows that $\CJ(\Q) \cong \Z^4$.
Computing canonical heights, either with an approach as in~\cite{MueSto}
or with the more general algorithms due independently to Holmes~\cite{Holmes}
and M\"uller~\cite{Mueller}, we find that an LLL-reduced basis of the
lattice $(G, \hat{h})$ is given by
\begin{gather*}
   P_1 = [(0,1) - \infty], \quad
   P_2 = [(1,1) - \infty], \quad
   P_3 = [(-1,1) - \infty], \\
   P_4 = [(1,-1) + (\omega,-1) + (\omega^2, -1) - 3 \cdot\infty]
\end{gather*}
with height pairing matrix
\[ M \approx  \begin{pmatrix}
                  0.17820 &  0.01340 & -0.05683 &  0.08269 \\
                  0.01340 &  0.81995 & -0.34461 & -0.26775 \\
                -0.05683 & -0.34461 &  0.98526 &  0.37358 \\
                  0.08269 & -0.26775 &  0.37358 &  1.07765
              \end{pmatrix} .
\]
We can bound the covering radius~$\rho$ of this lattice by
$\rho^2 \le 0.50752$. Using Example~\ref{Ex:diff},
it follows that if $G \neq \CJ(\Q)$, then there
must be a point $P \in \CJ(\Q) \setminus G$ satisfying
\[ h(P) \le \rho^2 + \beta \le 0.50752 + 3.90403 = 4.41155 , \]
so that we can write $\kappa(P) = (\xi_1 : \xi_2 : \ldots : \xi_8) \in \CK(\Q)$
with coprime integers~$\xi_j$
such that $|\xi_j| \le \lfloor e^{4.41155} \rfloor = 82$.
We can enumerate all points in~$\CK(\Q)$ up to this height bound and check
that no such point lifts to a point in~$\CJ(\Q)$ that is not in~$G$.
(Compare \cite[\S7]{StollH2} for this approach to determining the Mordell-Weil group.)
We have therefore proved the following.

\begin{Proposition} \label{P:ExMW}
  The group $\CJ(\Q)$ is free abelian of rank~$4$, generated by the points
  $P_1$, $P_2$, $P_3$ and~$P_4$.
\end{Proposition}

A Mordell-Weil sieve computation as described in~\cite{BruinStollMWS} shows
that any unknown rational point on~$\CC$ must differ from one of the eleven
known points
\[ \infty,\; (-1, \pm 1),\; (0, \pm 1),\; (\tfrac{1}{4}, \pm\tfrac{1}{64}),\;
   (1, \pm 1),\; (5, \pm 559)
\]
by an element of $B \cdot \CJ(\Q)$, where
\[ B = 2^6 \cdot 3^3 \cdot 5^3 \cdot 7^2 \cdot 11 \cdot 13 \cdot 17 \cdot 19 \cdot 23
          \cdot 29 \cdot 31 \cdot 37 \cdot 43 \cdot 47 \cdot 53 \cdot 61 \cdot 71 \cdot 79
          \cdot 83 \cdot 97
     \approx 1.1 \cdot 10^{32} .
\]
In particular, we know that every rational point is in the same coset
modulo~$2 \CJ(\Q)$ as one of the known points. For each of these cosets
(there are five such cosets: the points with $x$-coordinate~$1/4$ are
in the same coset as those with $x$-coordinate~$0$),
we compute a bound for the size of the $x$-coordinate of an integral
point on~$\CC$ with the method given in~\cite{BMSST}. This shows that
\[ \log |x| \le 2 \cdot 10^{1229} \]
for any such point~$(x,y)$. On the other hand, using the second stage of
the Mordell-Weil sieve as explained in~\cite{BMSST}, we obtain a lattice
$L \subset \Z^4$ of index \hbox{$\approx 2.3 \cdot 10^{2505}$} such that the minimal
squared euclidean length of a nonzero element of~$L$ is $\approx 2.55 \cdot 10^{1252}$
and such that every rational point on~$\CC$ differs from one of the known
points by an element in the image of~$L$ in~$\CJ(\Q)$ under the isomorphism
$\Z^4 \stackrel{\cong}{\to} \CJ(\Q)$ given by the basis above. This is more
than sufficient to produce a contradiction to the assumption that there
is an integral point we do not already know. We have therefore proved:

\begin{Theorem}
  The only points in~$\CC(\Q)$ with integral $x$-coordinate are
  \[ (-1, \pm 1),\; (0, \pm 1),\; (1, \pm 1),\; (5, \pm 559) \;. \]
  In particular, the only integral solutions of the equation
  \[ y^2 - y = x^7 - x \]
  are $(x,y) = (-1, 0)$, $(-1, 1)$, $(0, 0)$, $(0, 1)$, $(1, 0)$, $(1, 1)$,
  $(5, 280)$ and~$(5, -279)$.
\end{Theorem}

\section{Quadratic twists} \label{S:twists}

Let $F$ be a squarefree octic binary form over a field~$k$ not of characteristic~$2$
and let $c \in k^\times$. Then the Kummer varieties $\CK$ and~$\CK^{(c)}$
associated to~$F$ and to~$cF$, respectively, are isomorphic,
with an isomorphism from the former to the latter being given by
\[ (\xi_1 : \xi_2 : \xi_3 : \ldots : \xi_7 : \xi_8)
     \longmapsto (\xi_1 : c \xi_2 : c \xi_3 : \ldots : c \xi_7 : c^2 \xi_8) .
\]
We can therefore use~$\CK$ as a model for the Kummer variety associated
to the curve $\CC^{(c)} \colon y^2 = c F(x,z)$. This will in general change
the naive height of a point $P \in \CJ^{(c)}(\Q)$, but will not affect
the canonical height, which is insensitive to automorphisms of the ambient~$\BP^7$.
The duplication map is preserved by the isomorphism. This implies that the
height difference bounds of Lemmas~\ref{L:nonarch} and~\ref{L:arch1} for~$F$
apply to~$\CK$, even when $\CK$ is used as the Kummer variety of~$\CC^{(c)}$.
This is because these bounds are valid for all $k_v$-points on~$\CK$,
regardless of whether they lift to points in~$\CJ(k_v)$ or not.
Note, however, that the result of Lemma~\ref{L:regular} does \emph{not}
carry over: in the interesting case, $c$ has odd valuation at~$v$, and
so we are in effect looking at (certain) points on~$\CJ$ defined over a
ramified quadratic extension of~$k_v$. Since in terms of the original
valuation, the possible values of the valuation on this larger field are
now in $\frac{1}{2}\Z$, the argument in the proof of Lemma~\ref{L:regular}
breaks down.

When working with this model, one has to modify the criterion for a point to
lift to~$\CJ(k)$ by multiplying the $\mu_{ijk}$ by~$c$.

As an example, consider the curve given by
\[ \binom{y}{2} = \binom{x}{7} . \]
It is isomorphic to the curve
\[ \CC \colon y^2 = 70 (x^7 - 14 x^5 + 49 x^3 - 36 x + 630) = 70 F(x, 1) \]
where $F$ is the obvious octic binary form. The $2$-Selmer rank of its
Jacobian~$\CJ$ is~$9$, $\CJ(\Q)$ is torsion free, and the subgroup~$G$ of~$\CJ(\Q)$
generated by differences of the 27~small rational points on~$\CC$ has rank~$9$
with LLL-reduced basis
\begin{gather*}
  [(-2, 210) - \infty], \quad [(1, 210) - \infty], \quad [(3, 210) - \infty], \\
  [(2, 210) - \infty], \quad [(-3, 210) - \infty], \quad [(4, 630) - \infty], \\
  [(-\tfrac{5}{2}, -\tfrac{1785}{8}) + (3, 210) + (4, 630) - 3 \infty], \\
  [(0, 210) - \infty], \quad [(6, 3570) - \infty] .
\end{gather*}
We would like to show that these points are actually generators of~$\CJ(\Q)$.

Using the Kummer variety associated
to~$70 F$, we obtain the following bound for $\mu_v$ at the bad primes and infinity
(using the valuations of the resultants~$r(T)$, Lemma~\ref{L:regular}
and Lemma~\ref{L:arch1}):
\begin{gather*}
  \mu_2 \ge -6 \log 2, \quad \mu_3 \ge -\tfrac{10}{3} \log 3, \quad
  \mu_5 \ge -\tfrac{10}{3} \log 5, \quad \mu_7 \ge -\tfrac{8}{3} \log 7, \\
  \mu_{13} = 0, \quad \mu_{17} \ge -\tfrac{2}{3} \log 17 , \quad
  \mu_{15717742643} = 0, \quad \mu_{\infty} \ge -0.6152.
\end{gather*}
The resulting bound $\approx 20.88$ for $h - \hat{h}$ is \emph{much} too large to be useful.

However, using the Kummer variety associated to~$F$, we find that
\begin{gather*}
  \mu_2 \ge -\tfrac{10}{3} \log 2, \quad \mu_3 \ge -\tfrac{10}{3} \log 3, \quad
  \mu_5 \ge -\tfrac{2}{3} \log 5, \quad \mu_7 = 0, \\
  \mu_{13} = 0, \quad \mu_{17} \ge -\tfrac{2}{3} \log 17 , \quad
  \mu_{15717742643} = 0, \quad \mu_{\infty} \ge -0.6152.
\end{gather*}
This gives a bound of $\approx 9.55$ (now for a different naive height),
which is already a lot better, but still a bit too large for practical purposes.
Now one can check that for a point $P \in \CJ(\Q_p)$ with $p \in \{5, 17\}$,
we always have $\kappa(2 P) \in \CK_\good$. This implies that we get a
better estimate
\[ h(2P) \le \hat{h}(2P) + \tfrac{10}{3} \log 6 + 0.6152 \le \hat{h}(2P) + 6.588 \]
for $P \in \CJ(\Q)$. A further study of the situation at $p = 3$ reveals
that $\mu_3$ factors through the component group~$\Phi$ of the N\'eron model
of~$\CJ$ over~$\Z_3$, which has the structure $\Z/3\Z \times \Z/4\Z \times \Z/2\Z$,
and that the minimum of~$\mu_3$ on $2\Phi$ is $-\frac{5}{3} \log 3$.
This leads to
\begin{equation} \label{E:htbd}
  h(2P) \le \hat{h}(2P) + 4.757 .
\end{equation}
We enumerate all points~$P$ in~$\CJ(\Q)$ such that $h(P) \le \log 2000$
using a $p$-adic lattice-based approach with $p = 277$, as follows. For each
of the $10\,965\,233$~points
$\kappa(0) \neq Q \in \CK(\F_p)$ that are in the image of~$\CJ(\F_p)$,
we construct a sublattice~$L_Q$ of~$\Z^8$ such that for every point $P \in \CJ(\Q)$
such that $\kappa(P)$ reduces mod~$p$ to~$Q$, every integral coordinate vector
for~$\kappa(P)$ is in~$L_Q$ and such that $(\Z^8 : L_Q) \ge p^{11}$.
We then search for short vectors in~$L_Q$, thus obtaining all points of
multiplicative naive height $\le 2000$. Note that all these points are smooth
on~$\CK$ over~$\F_p$, since $\#\CJ(\F_p)$ is odd. This computation took about two
CPU weeks. For points reducing to the
origin, we see that the quadratic equation satisfied by points on~$\CK$
forces $\xi_1$ to be divisible by $p^2 > 2000$, so $\xi_1 = 0$, and every
such point must be on the theta divisor. A point $P = [P_1 + P_2 - 2 \cdot \infty] \in \CJ(\Q)$
reduces to the origin if and only if the points $P_1$ and~$P_2$ reduce
to opposite points; in particular, the polynomial whose roots are the $x$-coordinates
of $P_1$ and~$P_2$ reduces to a square mod~$p$. Since the coefficients are bounded
by $7 = \lfloor 2000/p \rfloor$, divisibility of the discriminant by~$p$
implies that the discriminant vanishes, so that $P_1 = P_2$, and the point~$P$
does not reduce to the origin, after all.

We find no point~$P$ such that $0 < \hat{h}(P) < \hat{h}(P_1) \approx 1.619$,
where $P_1$ is a known point of minimal positive canonical height,
and no points~$P$ outside~$G$ such that $\hat{h}(P) < 2.844 \approx \log 2000 - 4.757$.
Since the bound~\eqref{E:htbd} is only valid on~$2 \CJ(\Q)$, this implies that
there are no points $P \in \CJ(\Q)$ with $0 < \hat{h}(P) < 0.711 \mathrel{=:} m$.
Using the bound (see~\cite{FlynnSmart})
\[ I \le \left\lfloor\sqrt{\frac{\gamma_9^9 \det(M)}{m^9}}\right\rfloor \le 1787 \]
for the index of the known subgroup in~$\CJ(\Q)$,
where $\gamma_9$ denotes the Hermite constant for 9-dimensional lattices
and $M$ is the height pairing matrix of the basis of the known subgroup of~$\CJ(\Q)$,
we see that it suffices to rule out all primes up to~$1787$ as possible index
divisors. We therefore check that the known subgroup~$G$ is in fact saturated at all
those primes with the method already introduced in~\cite{FlynnSmart}:
to verify saturation at~$p$, we find sufficiently many primes~$q$ of good
reduction such that $\#\CJ(\F_q)$ is divisible by~$p$ (usually nine such primes
will suffice) and check that the kernel of the natural map
\[ G/pG \To \prod_q \CJ(\F_q)/p \CJ(\F_q) \]
is trivial. This computation takes a few CPU days; the most time-consuming
task is to find $\#\CJ(\F_q)$ for all primes~$q$ up to $q = 322\,781$
(which is needed for $p = 1471$). This gives the following result.

\begin{Theorem}
  The points $[P_j - \infty]$ freely generate~$\CJ(\Q)$, where the $P_j \in \CC(\Q)$
  are the points with the following $x$-coordinates and positive $y$-coordinate:
  \[ -3,\; -2,\; -\tfrac{5}{2},\; 0,\; 1,\; 2,\; 3,\; 4,\; 6 \;. \]
\end{Theorem}

In principle, one could now try to determine the set of integral points
on~$\CC$ with the method we had already used for $y^2 - y = x^7 - x$.
However, a Mordell-Weil sieve computation with a group of rank~9 is a
rather daunting task, which we prefer to leave to the truly dedicated reader.


\end{document}